\documentclass[preprint, noinfoline, twoside]{imsart}

\usepackage{amsthm,amsmath}
\doi{000}
\pubyear{0000}
\volume{0}
\firstpage{1}
\lastpage{1}
\startlocaldefs
	\theoremstyle{plain}
	\newtheorem{theorem}{Theorem}[section]
	
	\newtheorem{lemma}[theorem]{Lemma}
	
	\theoremstyle{remark}
	\newtheorem{definition}[theorem]{Definition}
	\newtheorem{example}[theorem]{Example}
	\newtheorem{assumption}[theorem]{Assumption}
	\RequirePackage{amsmath}
	\RequirePackage{amsfonts}
	\RequirePackage{amssymb}
	\RequirePackage{amsthm}
	\RequirePackage{mathrsfs}
	
	\RequirePackage{xargs}
	\RequirePackage{xparse}

	\usepackage[inline]{enumitem}
	
	\usepackage[margin=1cm, font={small}]{caption}
	\usepackage{subcaption}
	
	\usepackage{tikz}
	\usetikzlibrary{external}
	\usetikzlibrary{positioning}
	\usetikzlibrary{arrows}
	\usetikzlibrary{backgrounds}

	\usepackage{graphicx}
	\graphicspath{ {figures/} }
	
	\usepackage{lmodern}
	
	\usepackage{soul}
	
	\usepackage{algorithm}
	\usepackage{algorithmicx}
	\usepackage[noend]{algpseudocode}

	\usepackage{etoolbox}

	\usepackage{braket}
	
\RequirePackage[colorlinks,citecolor=blue,urlcolor=blue]{hyperref}	
	
	\RequirePackage{listings}
	
	\newcommand{\andtext}{\ \mathrm{and}\ }
	\newcommand{\ortext}{\ \mathrm{or}\ }
	\newcommand{\fullstop}{\text{.}}
	\newcommand{\comma}{\text{,}}
	
	\newcommand{\pvalue}{\(p\)-value}
	\newcommand{\ttest}{\(t\)-test}
	\newcommand{\ztest}{\(z\)-test}
	\newcommand{\textrootn}{root-\(n\)}

	\newcommand{\real}{\mathbb{R}}

	\newcommand{\MSE}{\operatorname{MSE}}

	\newcommand{\disjunion}{\uplus}
	
	\newcommand{\indep}{\perp \!\!\! \perp}
	\NewDocumentCommand{\convdist}{}{\rightsquigarrow}

	\newcommand{\asymleq}[1]{o\left( #1 \right)}
	\newcommand{\probasymleq}[1]{o_{\mathrm{P}}\left( #1 \right)}
	\newcommand{\asymgtr}[1]{\omega\left( #1 \right)}
	\NewDocumentCommand{\asymstocleq}{O{P_{n, \theta}}}{
		O_{#1}
	}
	\NewDocumentCommand{\asymstocless}{O{P_{n, \theta}}}{
		o_{#1}
	}

	\newcommand{\likelihood}[2]{\mathcal{L}\left({#1}\mid{#2} \right) }
	
	\NewDocumentCommand{\regcoefx}{m O{\ifnum #1=1 A \else M \fi}}{
		{
			\ifnum #1=1 \gamma \else \beta \fi
		}_{#2}
	}
	
	\ExplSyntaxOn
	\DeclareExpandableDocumentCommand{\IfNoValueOrEmptyTF}{mmm}
	{
		\IfNoValueTF{#1}{#2}
		{
			\tl_if_empty:nTF {#1} {#2} {#3}
		}
	}
	\ExplSyntaxOff
	\NewDocumentCommand{\regcoef}{m O{} O{}}{
		{
			\ifnum #1=1 \gamma \else \beta \fi
		}\IfNoValueOrEmptyTF{#2}{}{^{#2}}
		\IfNoValueOrEmptyTF{#3}{}{_{#3}}
	}
	
	\NewDocumentCommand{\maxregcoef}{O{} O{}}{
		\max \left\lbrace
		\left| \regcoef{1}[#1][#2] \right|,
		\left| \regcoef{2}[#1][#2] \right|
		\right\rbrace
	}
	\NewDocumentCommand{\minregcoef}{O{} O{}}{
		\min \left\lbrace
		\left| \regcoef{1}[#1][#2] \right|,
		\left| \regcoef{2}[#1][#2] \right|
		\right\rbrace
	}
	
	\newcommand{\regerror}[1]{\varepsilon^{ \ifnum #1=1 M \else Y \fi}}
	
	\newcommand{\estim}[1]{\avg{\regcoef{#1}}}
	\NewDocumentCommand{\maxestim}{}{
		\max \left\lbrace
		\left| \estim{1} \right|,
		\left| \estim{2} \right|
		\right\rbrace
	}
	\NewDocumentCommand{\minestim}{}{
		\min \left\lbrace
		\left| \estim{1} \right|,
		\left| \estim{2} \right|
		\right\rbrace
	}
	\NewDocumentCommand{\product}{O{\estim{1}}O{\estim{2}}}{{#1}\,{#2}}
	\newcommand{\mle}[1]{\widehat{#1}}
	
	\NewDocumentCommand{\estimreg}{O{n}}{T_{#1}}
	\NewDocumentCommand{\estimfil}{O{n}}{T^{\left(S \right)}_{#1}}
	\NewDocumentCommand{\medparam}{O{} O{}}{\theta^{#1}_{#2}}
	\NewDocumentCommand{\medparamzero}{O{}}{\theta^{#1}_{0}}

	\NewDocumentCommand{\medparamfull}{O{}}{\left( \regcoef{1}[#1], \regcoef{2}[#1] \right)}
	\NewDocumentCommand{\medparamspace}{}{\Theta}
	\NewDocumentCommand{\medparamnullspace}{}{\Theta_0}
	\NewDocumentCommand{\medparamnullsubspace}{m}{\Theta_{#1}}
	\NewDocumentCommand{\regpval}{m}{p^{\left(\regcoef{#1} \right) }}
	\NewDocumentCommand{\maxregpval}{}{
		\max \left\lbrace \regpval{1}, \regpval{2} \right\rbrace
	}
	
	\NewDocumentCommand{\reghyp}{m m}{H_{#1}^{\left(\regcoef{#2} \right) }}
	
	\NewDocumentCommand{\pjoint}{}{p_{\mathrm{joint}}}
	\NewDocumentCommand{\hypname}{m O{}}{H_{#1}^{#2}}
	\NewDocumentCommand{\filhyp}{m O{}}{\widetilde{H}_{#1}^{#2 }}

	\NewDocumentCommand{\basestat}{O{} O{}}{T^{#1}_{#2}}
	\NewDocumentCommand{\baserr}{O{} O{}}{R^{#1}_{#2}}
	\NewDocumentCommand{\baserradj}{O{} O{}}{\baserr[{#1} \IfNoValueOrEmptyTF{#1}{}{,} \mathrm{adj}][#2]}
	\NewDocumentCommand{\filtstat}{O{} O{}}{S^{#1}_{#2}}
	\NewDocumentCommand{\filtrr}{O{} O{}}{Q^{#1}_{#2}}
	\NewDocumentCommand{\filevent}{O{} O{}}{\filtstat[#1][#2] \notin \filtrr[#1][#2]}
	\NewDocumentCommand{\nfilevent}{O{} O{}}{\filtstat[#1][#2] \in \filtrr[#1][#2]}
	\NewDocumentCommand{\baserejevent}{O{} O{}}{\basestat[#1][#2] \in \baserr[#1][#2]}
	\NewDocumentCommand{\nbaserejevent}{O{} O{}}{\basestat[#1][#2] \notin \baserr[#1][#2]}
	\NewDocumentCommand{\baserejeventadj}{O{} O{}}{\basestat[#1][#2] \in \baserradj[#1][#2]}
	\NewDocumentCommand{\nbaserejeventadj}{O{} O{}}{\basestat[#1][#2] \notin \baserradj[#1][#2]}
	\NewDocumentCommand{\shrinkestim}{O{} O{}}{
		{\basestat[#1][#2]}^{\left( \filtstat[#1] \right) }
	}
	\NewDocumentCommand{\shrinkestimdef}{O{} O{}}{
		\left(\basestat[#1][#2]-\psi\left(\medparamzero \right)  \right) \indc{\nfilevent} + \psi\left(\medparamzero \right)
	}

	\NewDocumentCommand{\hypset}{}{\mathcal{H}}
	
	\newcommand{\nullcase}[1]{(#1)}
	
	\NewDocumentCommand{\hypnull}{O{H_0} m}{{#1}\!: #2}
	\NewDocumentCommand{\hypaltr}{O{H_1} m}{\hypnull[#1]{#2}}
	\NewDocumentCommand{\hypothesis}{O{H_0} m O{H_1} m}{
		\hypnull[#1]{#2} \quad \text{vs.} \quad \hypnull[#3]{#4}
	}

	\NewDocumentCommand{\NIE}{O{a, a^*} O{Y}}{\mathrm{NIE}_{\,#1}\left(#2\right)}
	\newcommand{\cond}[2]{\left.{#1}\,\middle|\,{#2}\right.}
	\newcommand{\filter}{*}
	
	\NewDocumentCommand{\lpspace}{O{p}}{\(\mathscr{L}^{#1}\)}
	\NewDocumentCommand{\lpnorm}{O{\cdot}O{p}}{\left\| {#1} \right\|_{#2} }
	
	\NewDocumentCommand{\contunifdist}{O{0}O{1}}{\operatorname{U}\left(#1, #2 \right) }
	
	\NewDocumentCommand{\meansym}{}{\mathrm{E}}
	\NewDocumentCommand{\lmeanbrkt}{}{[}
	\NewDocumentCommand{\rmeanbrkt}{}{]}
	\NewDocumentCommand{\mean}{m}{\meansym \left\lmeanbrkt {#1} \right\rmeanbrkt}
	\NewDocumentCommand{\var}{m}{\operatorname{var}\left( {#1} \right) }
	\newcommandx{\prob}[3][1=\theta, 2={n,}, usedefault=@]{P_{#2 #1}{\left( #3 \right) }}
	\NewDocumentCommand{\avg}{m}{\overline{#1}}
	\NewDocumentCommand{\nordist}{O{0} O{1}}{\operatorname{N}\left( {#1},{#2} \right) }
	\NewDocumentCommand{\indc}{s m}{
		\mathbf{1}_{\IfBooleanTF{#1}{#2}{\left[{#2}\right]}}
	}
	\newcommand{\transpose}[1]{#1^\mathsf{T} }
	
	\numberwithin{equation}{section}
	\numberwithin{table}{section}
	\numberwithin{figure}{section}
	
 	\usepackage[authoryear]{natbib}
\makeatletter
\newcommand*{\algrule}[1][\algorithmicindent]{\makebox[#1][l]{\hspace*{.5em}\thealgruleextra\vrule height \thealgruleheight depth \thealgruledepth}}%
\newcommand*{\thealgruleextra}{}
\newcommand*{\thealgruleheight}{.75\baselineskip}
\newcommand*{\thealgruledepth}{.25\baselineskip}

\newcount\ALG@printindent@tempcnta
\def\ALG@printindent{%
	\ifnum \theALG@nested>0% is there anything to print
	\ifx\ALG@text\ALG@x@notext% is this an end group without any text?
	\else
	\unskip
	\addvspace{-1pt}% FUDGE to make the rules line up
	\ALG@printindent@tempcnta=1
	\loop
	\algrule[\csname ALG@ind@\the\ALG@printindent@tempcnta\endcsname]%
	\advance \ALG@printindent@tempcnta 1
	\ifnum \ALG@printindent@tempcnta<\numexpr\theALG@nested+1\relax% can't do <=, so add one to RHS and use < instead
	\repeat
	\fi
	\fi
}%

\patchcmd{\ALG@doentity}{\noindent\hskip\ALG@tlm}{\ALG@printindent}{}{\errmessage{failed to patch}}
\makeatother

\newbox\statebox
\newcommand{\myState}[1]{%
	\setbox\statebox=\vbox{#1}%
	\edef\thealgruleheight{\dimexpr \the\ht\statebox+1pt\relax}%
	\edef\thealgruledepth{\dimexpr \the\dp\statebox+1pt\relax}%
	\ifdim\thealgruleheight<.75\baselineskip
	\def\thealgruleheight{\dimexpr .75\baselineskip+1pt\relax}%
	\fi
	\ifdim\thealgruledepth<.25\baselineskip
	\def\thealgruledepth{\dimexpr .25\baselineskip+1pt\relax}%
	\fi
	\State #1%
	\def\thealgruleheight{\dimexpr .75\baselineskip+1pt\relax}%
	\def\thealgruledepth{\dimexpr .25\baselineskip+1pt\relax}%
}
\let\originalleft\left
\let\originalright\right
\renewcommand{\left}{\mathopen{}\mathclose\bgroup\originalleft}
\renewcommand{\right}{\aftergroup\egroup\originalright}
\bibpunct{(}{)}{,}{a}{,}{,}
\endlocaldefs

\begin{document}

\begin{frontmatter}

% "Title of the Paper"
\title{Improving Efficiency of Tests for Composite Null Hypotheses}%\thanksref{t1}}
% TODO thanks?
%\thankstext{t1}{This is an original survey paper}
\runtitle{Improving efficiency of tests for composite null hypotheses}

% indicate corresponding author with \corref{}
% \author{\fnms{John} \snm{Smith}\thanksref{t2}\corref{}\ead[label=e1]{smith@foo.com}\ead[label=e2,url]{www.foo.com}}
% \thankstext{t2}{Thanks to somebody} 
% \address{line 1\\ line 2\\ \printead{e1}\\ \printead{e2}}

\author{
	\fnms{Yotam} \snm{Leibovici} \ead[label=e1]{yotaml@technion.ac.il} \and{}
	\fnms{Yair} \snm{Goldberg} \ead[label=e2]{yairgo@technion.ac.il}
}
\address{
	Faculty of Industrial Engineering and Management,\\
	Technion---Israel Institute of Technology, Israel\\
	\printead{e1,e2}
}

\runauthor{Y.\ Leibovici and Y.\ Goldberg}

% TODO shorter abstract
\begin{abstract}
The goal of mediation analysis is to study the effect of exposure on an outcome interceded by a mediator. Two simple hypotheses are tested: the effect of the exposure on the mediator, and the effect of the mediator on the outcome. 
When either of these hypotheses is true, a predetermined significance level can be assured. When both nulls are true, the same test becomes conservative. Adaptively finding the correct scenario enables customizing the tests and consequently enlarges their efficiency, which is most important in a multiple testing framework.  In this work, we link between adaptive two-stage procedures and shrinkage estimators. We first study the properties of shrinkage estimators, and characterize their behavior at different parameter points using local asymptotics. We formulate theoretical results regarding shrinkage estimators, compared to regular estimators. We then discuss the multiple-testing framework and state results about using shrinkage estimator in  two-stage procedures for controlling the FWER. Taking advantage of these theoretical results, we suggest a number of estimators and test statistics for the two-stage mediation procedures. We then investigate their empirical FWER and power, compared to regular estimators and tests, through simulations. \end{abstract}



% history:
% \received{\smonth{1} \syear{0000}}

%\tableofcontents

\end{frontmatter}

% Main text entry area

\section{Introduction}

%\subsection{Motivation}

% Causal inference is a theory of drawing conclusions about causal connections between phenomena. Specifically, the dependency of an effect on the occurrence of some conditions.

Mediation analysis, a research area within the field of causal inference, is a collection of methods employed for studying the causal effect of an exposure on an outcome, interceded by a mediator.
%Mediation can be utilized in investigating a variety of situations. For instance, in epidemiology, one of which is obesity.

For instance, obesity has already been found to be correlated to poor socioeconomic level in childhood \citep{senese_associations_2009}, as well as to changes in DNA methylation process \citep{borghol_associations_2012, agha_adiposity_2015}. \citet{huang_genome-wide_2019} hypothesized that socioeconomic disadvantages in childhood alter DNA methylation, which in turn affects obesity in adulthood, and used mediation methods for testing this hypothesis. This is a typical example of the mediation process: the hypothesis tested whether the outcome (obesity) is affected by the exposure (socioeconomic level) through a mediator (DNA methylation).

%Previous works \citep{barfield_testing_2017,huang_genome-wide_2019} indicate that standard methods for testing hypotheses related to mediation work well under some assumptions, but tend to lose efficiency in some sense under others. As we explain below, situations in which testing hypotheses related to mediation are inefficient become more and more relevant. In the following, we propose an approach to improve efficiency in these situations.

%\subsection{Power Loss} \label{sec:power-loss}

In mediation analysis, two hypotheses are tested: the effect of the exposure on the mediator, and the effect of the mediator on the outcome. Finding both of them as true is accepting the hypothesis that the influence of the exposure on the outcome is mediated, and that the mediator is essential for understanding the connection between them. Hence, we are interested in a composite hypothesis which contains both of the preceding hypotheses. Some statistical literature refers to this setting of composite hypotheses, as the \emph{intersection-union test} \citep{berger_bioequivalence_1996, liu_uniformly_1995, berger_multiparameter_1982%%,doi_introduction_1997,brunner_intersection-union_2004
}.% but in the case presented here, not all issues have been resolved. %therefore call for further consideration and analysis.

The null hypothesis of mediation can be divided into three cases: The first is when the mediator is not significantly influenced by the exposure but significantly influence the outcome;
% in and of itself
the second is when the mediator is significantly influenced by the exposure but it does not significantly influences the outcome;
% in and of itself
and the third case is when both influences are insignificant.
%In traditional epidemiology research, plausible reasons make the studies focus on Cases 1 and 2, while Case 3 is not frequently encountered.
Standard tests for the hypothesis, such as the \emph{joint significance test}, which rejects the hypothesis if both influences were found as significant \citep{mackinnon_comparison_2002}, may be optimal under the first two cases, but become much more conservative under the third, as their significance level decreases. In genome-wide epidemiological research, %gaining support from increase in genetic data collection and the possibility of its study as mediating processes,
when the number of hypotheses can be extremely large, the third case emerges more often. Consequently, employing more adaptive methods of analyzing mediation data is crucial for efficient research \citep{barfield_testing_2017}. 

%\subsection{Adaptive Learning and Super-Efficiency}

One way to handle the power loss in the third case, is to use two-stage testing procedures \citep{ignatiadis_data-driven_2016, djordjilovic_optimal_2020}. Consider a multiple testing problem over a common parameter space. Assume that the first and the second cases of the null are composite hypotheses while the third case is a simple hypothesis consisting of a single parameter point. Assume also that for each hypothesis there is a test which is efficient for the first two cases, but not for the third case. We refer to this test as the \emph{base test}. % This kind is proposed as an approach of handling the issues that occur in the third case, in order to figure out the current case we are at. Finding adaptively the current case we are at enables us to customize our tests, and consequently enlarge their efficiency.
%The general procedure is based on a given test, which we refer to as the \emph{base test}.

The two-stage procedure is defined as follows. In the first stage, for every hypothesis, a stricter test called the \emph{filtration test} is performed. Its aim is to filter, or retain, null hypotheses that are easy to identify, specifically the third case hypotheses. %Specifically, the filtration test is defined by a set such that when the filtration test statistic
In the second stage, the original (base) test, up to an adjustment, is performed for the unfiltered hypotheses.
This adjustment depends not only on the filtration test itself, but also on the size of the unfiltered hypotheses set.

The asymptotic behavior of this two-stage procedure needs to be considered cautiously. In practice, depending on the parameter values under the null, one may not be able to distinguish between, say, the first and the third cases. Thus, one has to take into account the local asymptotic behavior, which can vary across different parameter points and different kinds of filtration tests. In fact, the commonly used terminology to describe the parameter points' characterization is discrete and composed of the three aforementioned cases of the null and the alternative case. We use a continuous terminology to describe the parameter points. Specifically, we consider sequences of null hypotheses which belong, for instance, to the first case but converge to the third. We also consider sequences of filtration tests, which, depending on the convergence rate, may or may not be able to distinguish between the first and the third case. % or  considering filtration probabilities and asymptotic distance from a point.
This approach is more appropriate for this context, in which there are infinite null parameter points, which are distinct from each other by their local geometry and asymptotic behavior, as well as alternatives that can approach different null points with various rates.

In order to get insights about this two-stage procedure, we first present the connection between filtration tests and shrinkage estimators. Indeed, the filtration test retains the null if it is likely that the value of the test statistic is a function of observations which are distributed according to the single parameter point of the third case. %The goal of the filtration step is to assess whether retains some hypotheses by their test-statistic values, 
This in fact means that the null is retained for all parameter points which are closed, in some sense, to the third case. However, when the observations are drawn from parameter points that are not closed to the third case, the hypothesis is not filtered, and thus the base test, up to adjustment, is performed. This is related to shrinkage estimators, where the estimator is either shrinked to a specific point, when it is likely that the observations were distributed according to this parameter point, but otherwise it remains as is.% to the null parameter point, when the filtration occurs. The obtained statistic has a form of \emph{shrinkage estimator}.

The canonical expample of shrinkage estimator is the well-known Hodges' estimator \citep{le_cam_asymptotic_1953}. For a location parameter model (e.g., normal distribution with unknown mean), Hodges' estimator aims to identify more efficiently whether the true parameter point is \(0\) or not. Asymptotically, the variance of the Hodges' estimator is lower than the Cram\'er-Rao bound at \(0\), which was considered to be the optimal that can be achieved. Such estimators are called \emph{super-efficient}. \citet{le_cam_asymptotic_1953} then showed that the Hodges' estimator has various undesirable properties in a neighborhood of the parameter where the super-efficiency was obtained.

In this work we study the properties of shrinkage estimators, and characterize their behavior in different parameter points. We formulate theoretical results about filtration estimators, compared to regular estimators, in different scenarios of convergence rates and filtration probabilities, under some regularity conditions. We return to the multiple-testing framework in order to state results about the two-stage tests and controlling the FWER. We suggest a number of estimators and test statistics for the two-stage mediation procedures. We then investigate their empirical FWER and power, compared to regular estimators and tests, through simulations.

This work is organized as follows. In Section~\ref{sec:med}, we describe the mediation test setting and discuss the power loss challenge. In Section~\ref{sec:filt}, the filtration tests framework is formulated, and linked to shrinkage estimators. In Section~\ref{sec:rates}, we define efficiency in this setting and investigate convergence rates of estimators depending on the convergence of the parameter points. In Section~\ref{sec:la}, we present theoretical results about these estimators and the two-stage testing procedure. In Section~\ref{appmc}, we perform simulation study of the two-stage procedure on the multiple-testing framework with various estimators. Proofs that were omitted from the main paper are provided in the Appendix.

\section{Testing for Mediation} \label{sec:med}

\subsection{Definitions and Assumptions} \label{sec:med-defs}

Being one of the main tools in causal analysis, mediation analysis has become a statistical subject of major interest. In the basic setting, one considers a causal system, which includes (see Figure \ref{fig:causalsystem}):
\begin{enumerate*}[label=(\alph*)]
	\item an exposure $A$;
	\item an outcome $Y$ affected by the exposure;
	\item a mediator, i.e., a variable $M$ added to the system, which is examined for how significantly it is infleunced by $A$ and influences $Y$; and
	\item a vector of baseline covariates $X$, which includes every other variable that may possibly have an influence on some of the variables above.
\end{enumerate*}
Having been revealed as a significant part of the system, the mediator might give a better perspective for statistical inference.

\begin{figure}
	\centering
	
	\begin{tikzpicture}[->,>=stealth',shorten >=1pt,auto,node distance=3cm,
	thick,main node/.style={circle,fill=blue!20,draw,font=\sffamily,
	}]
	
	\node[main node, fill=green!20] (X) at (0,0)  {X};
	\node[main node, fill=blue!20]  (A) at (4,1)  {A};
	\node[main node, fill=red!20]   (M) at (6,-1) {M};
	\node[main node, fill=blue!20]  (Y) at (8,1)  {Y};
	
	\path[every node/.style={font=\sffamily\small,
		,inner sep=1pt}]
	% Right-hand-side arrows rendered from top to bottom to
	% achieve proper rendering of labels over arrows.
	(X) edge [dashed] node {} (A)
	edge [dashed] node {} (M)
	edge [dashed, bend left=40] node{} (Y)
	(A) edge node {} (M)
	edge node {} (Y)
	(M) edge node {} (Y);
	
	% Left-hand-side arrows rendered from bottom to top to
	% achieve proper rendering of labels over arrows.
	\end{tikzpicture}
	
	\caption[Causal system for mediation]{Causal system with exposure $A$ affecting outcome $Y$, a mediator $M$ and an underlying vector of covariates $X$.} \label{fig:causalsystem}
\end{figure}
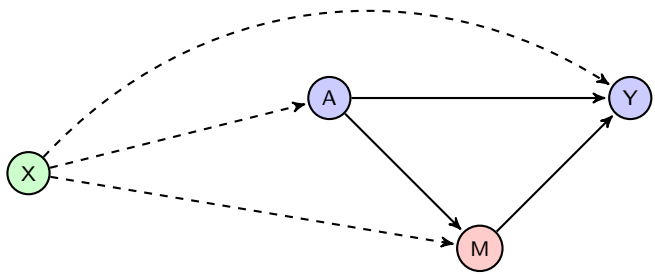

Formally, consider the following linear model for the connection between the variables:
\begin{equation} \label{eq:medregmodel}
	\begin{split}
		M &= \regcoefx{1}[0] + \transpose{\regcoefx{1}[X]} X + \regcoefx{1}[A] A + \regerror{1} \text{,} \\
		Y &= \regcoef{2}[][0] + \transpose{\regcoefx{2}[X]} X + \regcoefx{2}[A] A + \regcoefx{2}[M] M + \regerror{2} \text{.}
	\end{split}
\end{equation}
The mediation hypothesis concerns whether or not the mediator \(M\) explains significantly the causal connection between the exposure \(A\) and the outcome \(Y\). The extent to which the mediator effects this connection is formulated by the \emph{natural indirect effect}, NIE, whose definition requires an additional notation.

For two random variables, \(B\) and \(C\), where \(C\) is affected by \(B\), the potential outcome \(C_b\) is used for the random variable that its possible values correspond to the possible values \(C\) would attain had a variable \(B\) had been set to a value \(b\) \citep{pearl_direct_2001, pearl_causal_2009}. The NIE is defined as
\[\NIE = \mean{Y_{aM_{a^*}} - Y_{a^*M_{a^*}}} \comma \] 
which is the expected change in the value of \(Y\), when the value of the exposure \(A\) is fixed to \(a^*\), and the value of the mediator \(M\) changes its value from \(M_a\) to \(M_{a^*}\) \citep{robins_identifiability_1992, pearl_direct_2001,vanderweele_conceptual_2009}. Defined that way, the NIE is nonzero exactly when a change in the exposure would make such change in the mediator that in turn would change the effect, therefore it formally conveys our notion of mediation \citep{vanderweele_three-way_2013}.

The mediation hypothesis can now be phrased formally as
\[\hypothesis{\NIE = 0}{\NIE \neq 0} \fullstop\]
Note that \(Y_{aM_{a^*}}\) and \(Y_{a^*M_{a^*}}\) are \emph{counterfactual quantities}, i.e., they are not observed directly, thus in many cases the NIE cannot be evaluated from data with no additional assumptions. Consider the following four assumptions about the model, illustrated in Figure \ref{fig:forbidden}, for all possible values \(a\), \(a^*\) and \(m\) of \(A\) and \(M\), respectively. Here, the independence of $B$ and $C$ conditional on $D$ is denoted by $\cond{B \indep C}{D}$.
\begin{enumerate}[label = (\roman*)] \label{assumptions}
	\item $Y_a \indep  A \mid X$: No unknown confounders for the exposure-outcome relationship.
	\item $Y_{am} \indep M \mid  A,X  $: No unknown confounders for the mediator-outcome relationship, conditional on the exposure.
	\item $M_a \indep A \mid X$: No unknown confounders for the exposure-mediator relationship.
	\item $Y_{am} \indep M_{a^*} \mid X$: No unknown confounders for the mediator-outcome relationship, that itself is affected by exposure.
\end{enumerate}
 Under those assumptions, the NIE can be evaluated using the coefficients in the regression models, as \(\NIE = \regcoefx{1} \regcoefx{2} \left( a - a^* \right) \) \citep{pearl_direct_2001, vanderweele_three-way_2013, barfield_testing_2017}. With \(a \neq a^*\), this reduces the mediation hypothesis to
 \begin{equation} \label{eq:medhyp}
 	\hypothesis{
 		\regcoefx{1} \regcoefx{2} = 0
 	}{
 		\regcoefx{1} \regcoefx{2} \neq 0
 	} \fullstop
 \end{equation}

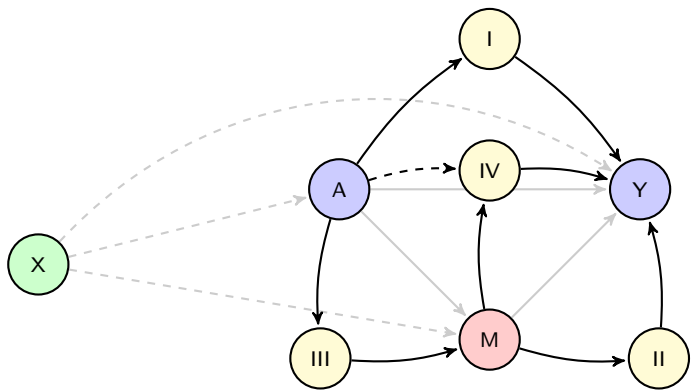
\begin{figure}
	\centering
	\begin{tikzpicture}[->,>=stealth',shorten >=1pt,auto,node distance=3cm,
	thick,main node/.style={circle,fill=blue!20,draw,font=\sffamily,minimum size=8mm,
	}]
	
	\node[main node, fill=green!20] (X) at (0,0)  {X};
	\node[main node, fill=blue!20]  (A) at (4,1)  {A};
	\node[main node, fill=red!20]   (M) at (6,-1) {M};
	\node[main node, fill=blue!20]  (Y) at (8,1)  {Y};
	\node[main node, fill=yellow!20]  (3) at (3.75,-1.25)  {III};
	\node[main node, fill=yellow!20]  (2) at (8.25,-1.25)  {II};
	\node[main node, fill=yellow!20]  (1) at (6,3)  {I};
	\node[main node, fill=yellow!20]  (4) at (6,1.25)  {IV};
	
	\begin{scope}[on background layer]
	\path[thick,every node/.style={font=\sffamily\small
		,inner sep=1pt}]
	% Right-hand-side arrows rendered from top to bottom to
	% achieve proper rendering of labels over arrows.
	(X) edge [dashed, black!20] node {} (A)
	edge [dashed, black!20] node {} (M)
	edge [dashed, bend left=40, black!20] node{} (Y)
	(A) edge[black!20] node {} (M)
	edge [black!20] node {} (Y)
	edge[bend right=10] node{} (3)
	edge[bend left=10] node{}(1)
	edge [dashed,bend left=10] node {} (4)
	(M) edge[black!20] node {} (Y)
	edge[bend right=10] node{} (2)
	edge[bend left=10] node{} (4)
	(3) edge[bend right=10] node{} (M)
	(2) edge[bend right=10] node{} (Y)
	(1) edge[bend left=10] node{} (Y)
	(4) edge[bend left=10] node{} (Y)
	
	;
	\end{scope}
	% Left-hand-side arrows rendered from bottom to top to
	% achieve proper rendering of labels over arrows.
	\end{tikzpicture}
	\caption[Causal system assumptions]{Assumptions (i)--(iv) in Section \ref{assumptions} mean that connections through variables I, II, III and IV, respectively, do not exist.} \label{fig:forbidden}
\end{figure}

\subsection{Conservativity in Composite Hypotheses} \label{sec:med-powerloss}

Let \(\regcoef{1} = \regcoefx{1}\) and \(\regcoef{2} = \regcoefx{2}\) be the coefficients in the regression equations from the mediation model in \eqref{eq:medregmodel}. As mentioned, a test for the mediation effect, formulated as NIE nonzero, is equivalent to a test for the hypothesis in \eqref{eq:medhyp}.
Put in statistical terms, for the parameter of interest \(\medparam = \medparamfull \), the full parameter space is \(\medparamspace = \real^2\), whereas the null space of this hypothesis is \(\medparamnullspace = \Set{\medparamfull \in \real^2 | \regcoef{1} \regcoef{2} = 0 }\). The null space \(\medparamnullspace\) can be naturally divided into three complementary sub-spaces, \(\medparamnullspace = \medparamnullsubspace{00} \disjunion \medparamnullsubspace{01} \disjunion \medparamnullsubspace{10}\) (where \(\disjunion\) stands for disjoint union),
\begin{align} \label{eq:threecases}
\begin{split}
\Theta_{00} &= \Set{\medparamfull \in \Theta_0 | \regcoef{1} = 0 \andtext \regcoef{2} = 0} \text{,} \\
\Theta_{01} &= \Set{\medparamfull \in \Theta_0 | \regcoef{1} = 0 \andtext \regcoef{2} \neq 0} \text{,} \\
\Theta_{10} &= \Set{\medparamfull \in \Theta_0 | \regcoef{1} \neq 0 \andtext \regcoef{2} = 0} \fullstop
\end{split}
\end{align}
We refer to the case that the true parameter \(\medparam\in\medparamnullsubspace{ij}\) as Case~\nullcase{\(ij\)}, for \(i, j = 0, 1\). Cases~\nullcase{01} and~\nullcase{10} were mentioned in the introduction as the first and the second cases, whereas Case~\nullcase{00} was mentioned as the third case.

Let $\mle{\regcoef{1}}$ be a linear-regression estimator for the parameter $\regcoef{1}$ and let \(\mle{\regcoef{2}}\) be an estimator for $\regcoef{2}$. From the above four assumptions, one can show that the errors \(\regerror{1}\) and \(\regerror{2}\) are independent \citep{huang_genome-wide_2019}. A natural point estimator for the function of the parameter which appears in \eqref{eq:medhyp}, \(\psi = \psi\left(\theta \right)  = \psi\left(\regcoef{1}, \regcoef{2} \right) = \regcoef{1} \regcoef{2}\), would be \(T = \mle{\regcoef{1}} \mle{\regcoef{2}}\). Similarly, the same \(T\) can be used as a test statistic for the hypothesis \eqref{eq:medhyp}. This product estimator's asymptotic behavior depends on which of the sub-spaces of the null contains the true parameter. As explained below, when \(\psi = 0\), the rate of convergence of \(T\) is root-\(n\) for Cases \nullcase{01} and~\nullcase{10}, but is \(n\) for Case~\nullcase{00}.

Treating each of the parameters \(\regcoef{1}\) and \(\regcoef{2}\) separately and considering the two simple hypotheses,
\begin{align}
\label{eq:param1hyp}
\hypothesis[\reghyp{0}{1}]{\regcoef{1} = 0&}[\reghyp{1}{1}]{\regcoef{1} \neq 0} \comma\\
\label{eq:param2hyp}
\hypothesis[\reghyp{0}{2}]{\regcoef{2} = 0&}[\reghyp{1}{2}]{\regcoef{2} \neq 0 \comma}
\end{align}
yields to another approach. Rephrased as \[\hypothesis{\regcoef{1} = 0 \ortext \regcoef{2} = 0}{\regcoef{1} \neq 0 \andtext \regcoef{2} \neq 0} \comma\] hypothesis \eqref{eq:medhyp} becomes a special case of \emph{intersection-union hypotheses}:
\begin{equation} \label{eq:unionhyp}
H_0 = H_0^{\left( \regcoef{1} \right)} \cup H_0^{\left( \regcoef{2} \right)} \fullstop
\end{equation}
Here, the null hypothesis \(H_0\) is true if and only if at least one of the two simple hypotheses is true.

It is possible to test the hypothesis \eqref{eq:medhyp}, decomposed as union hypothesis in \eqref{eq:unionhyp}, using the \emph{joint significance test}. Let $\regpval{1}$ and $\regpval{2}$ denote the \pvalue{}s for the hypotheses in \eqref{eq:param1hyp} and \eqref{eq:param2hyp}, respectively, based on the estimators \(\mle{\regcoef{1}}\) and \(\mle{\regcoef{2}}\). When \(\regcoef{1} = 0\), then $\regpval{1} \sim\contunifdist $, while when \(\regcoef{1} \neq 0\), we have $\regpval{1} \to 0$, assuming consistency of \(\mle{\regcoef{1}}\), as the sample size increases. A natural statistic for the union test is \(p_{\mathrm{joint}}=\max \left\lbrace \regpval{1},\regpval{2} \right\rbrace \). Under Case~\nullcase{01}, and by symmetry also under Case~\nullcase{10}, \(\pjoint\) is distributed approximately \(\contunifdist\). Indeed, for \(\alpha \in \left(0,1 \right) \), \[P(\pjoint\le\alpha) = P(\regpval{1}\le\alpha)P(\regpval{2}\le\alpha) \to 1 \cdot \alpha = \alpha \comma\]
where we used the independence of \(\regpval{1}\) and \(\regpval{2}\).
However, under Case~\nullcase{00},
\[P(\pjoint\le\alpha) = P(\regpval{1}\le\alpha)P(\regpval{2}\le\alpha) \to \alpha \cdot \alpha = \alpha^2 \fullstop\]

In our framework, maintaining a significance level of $\alpha$ would force rejection of the joint-significance test only for \pvalue s that are less than the threshold $\alpha$, for Cases \nullcase{01} and \nullcase{10}. However, if we knew in advance that we are under Case~\nullcase{00} is the only possible option in the null, we could adjust the threshold to a larger value, $\sqrt{\alpha}$, and as a result amplify the power in that case. An oracle revealing the current case %, should one were available, and would reveal the current case,
would allow us to optimally customize our test and maximize the power, at least asymptotically. %(see Figure \ref{fig:power-loss}).

\section{Filtration} \label{sec:filt}
So far we discussed a single mediation problem, which yields a single hypothesis testing. In this section, we consider the mediation problem in a multiple-testing framework.

\subsection{The Filtration Algorithm} \label{sec:filt-alg}

In general, given a set of hypotheses, filtration is a preliminary test (the \emph{filtration test}) performed on every hypothesis ahead of the main test (the \emph{base test}), in order to determine which hypotheses are likely to belong to the null. The idea behind the filtration is that some null hypotheses may behave in a way that will ease identifying them, thus adjusting the criteria for them to be rejected, and hopefully increase power. Filtration is a popular method of increasing power in multiple hypotheses testing \citep{mcclintick_effects_2006, talloen_ini-calls_2007, hackstadt_filtering_2009}.

A filtration procedure is a two-stage process, which might have the form detailed in Algorithm~\ref{alg:filtration} \citep{bourgon_independent_2010, ignatiadis_data-driven_2016}. The algorithm is described for \(\hypset\), a collection of \(m\) hypotheses, with \(\basestat[i]\) and \(\baserr[i]\) as test statistic and rejection region, respectively, for the base test of the \(i\)-th hypothesis, and \(\filtstat[i]\) and  \(\filtrr[i]\) as test statistic and rejection region for the filtration test.

\begin{algorithm}
	\caption{Filtration} \label{alg:filtration}
	%\begin{spacing}{1.2}
	\begin{algorithmic}[0]
		\State \(\hypset \gets \left\lbrace 1, \dots, m \right\rbrace \)
		\Comment{The hypotheses set}
		\ForAll{\(i \in \hypset\)}
		\If{\(\filevent[i]\)}
		\State Retain \(\hypname{0}[i]\)
		\State \(\mathcal{H} \gets \mathcal{H} \smallsetminus \left\lbrace i \right\rbrace \)
		\Comment{Filter the improbable \(i\)-th hypothesis}
		\EndIf
		\EndFor
		\State \textsc{Adjust}\(\left( \left\lbrace R^i \right\rbrace _{i \in \mathcal{H}}\right) \)
		\Comment{To control the FWER}
		\ForAll{\(i \in \hypset\)}
		
		\If{\(\baserejeventadj[i]\)}
		\State Reject \(\hypname{0}[i]\)
		\Else
		\State Retain \(\hypname{0}[i]\)
		\EndIf
		\EndFor
	\end{algorithmic}
	%\end{spacing}
\end{algorithm}

\subsection{Shrinkage Estimators} \label{sec:filt-shrinkage}

In Algorithm \ref{alg:filtration}, fix a specific hypothesis of index \(i\), and omit the index for convenience. Consider a base test statistic \(\basestat\) which is an estimator for \(\psi\left(\theta \right) \), and a filtration test statistic \(\filtstat\) with rejection region \(\filtrr\), which, depending on filtering or un-filtering event, indicates the distance from the true parameter \(\psi\left(\theta \right) \) to the predefined point \(\psi\left(\theta_0 \right) \) (in the mediation hypothesis, \(\theta_0 = \left(0,0 \right) \)). A point-estimation equivalent of this two-stage process can be phrased as
\begin{equation} \label{eq:filtestimator}
	\shrinkestim = \left(T - \psi\left(\theta_0\right)\right) \indc{\nfilevent} + \psi\left( \theta_0\right) \fullstop
\end{equation}
The estimator \eqref{eq:filtestimator} encodes inside itself the two-stage process of Algorithm~\ref{alg:filtration}. Indeed, if the filtration test was not rejected (i.e., a filtration event occurred), the estimator's value becomes \(\psi\left(\theta_0 \right) \) for \(\theta_0 \in \Theta_0\), and when the filtration test was rejected, \eqref{eq:filtestimator} remains the same as the original \(\basestat\). Thus, \(\shrinkestim\) can be used as a test statistic to the test with an adjusted rejection region \(\baserradj\). Note that the obtained estimator is a so-called \emph{shrinkage estimator}. As an example, consider the base-test statistic \(T = \product\) and the filtration test \(\filevent\) as \(\left| \product \right| < cn^{-\delta} \), for a constant \(\delta > 0 \). The estimator obtained is \(\product \,\indc{\left| \product \right| > cn^{-\delta} }\).

Shrinkage estimators were originally produced in order to improve performance over standard estimators, usually in a specific parameter point, obtained by ``stretching'' the values of estimator towards a fixed parameter point. In many cases, the obtained improvement depends on the performance measurement (e.g., the loss function used) and might be problematic at a closer look.

A possible form of a shrinkage estimator, based on an existing estimator \(\basestat\) for \(\psi\left( \theta \right) \), and some shrinkage function \(A_{\theta_0}\), can be written as
\begin{equation} \label{eq:shrinkage}
	S_{\theta_0}\left(T\right) = \left(T - \psi\left(\theta_0 \right)  \right) A_{\theta_0}\left(T\right)  + \psi\left(\theta_0 \right) \comma
\end{equation}
which generalizes the form in \eqref{eq:filtestimator}. Two well-known examples of shrinkage estimators which fit into this form are James-Stein estimator \citep{kotz_estimation_1992} and Hodges' estimator \citep{le_cam_asymptotic_1953, van_der_vaart_asymptotic_1998}.

One of the versions of Hodges' estimator is defined as follows. Consider a location model, specifically normal distribution \(\nordist[\mu][1]\) with unknown mean and a fixed variance. Let $X_1,\dots,X_n$ be a sequence of i.i.d.\ observations from the distribution, and define, for estimating the expectation $\mu$,
\begin{equation} \label{eq:hodges}
	\widetilde{\mu} = \avg{X} \indc{\left| \avg{X}\right| >n^{-1/4}} \fullstop
\end{equation}
It can be shown \citep{van_der_vaart_asymptotic_1998} that while at every parameter point \(\mu \neq 0\) the estimator \(\widetilde{\mu}\) behaves asymptotically similar to the regular \(\avg{X}\) with root-\(n\) rate of convergence, when \(\mu = 0\), \(\widetilde{\mu}\) converges to \(\mu = 0\) with an ``arbitrarily fast" rate. The \emph{case} \(\mu = 0\) is now distinguished from all the others.

The idea behind the shrinkage estimators we discussed and showed in \eqref{eq:filtestimator} is inspired from the idea behind Hodges' estimator, and shrinkage estimators generally. In both situations, a regular estimator is manipulated in order to use inference about the current case of the parameter, and by that improve estimation. 
Unfortunately, the superiority of Hodges' estimator at $\mu=0$ comes with shortcomings at parameter points in its neighborhood. For the sake of formulating this idea, we shall use the local asymptotic framework of statistical experiments and their limits.

%Indeed, by the CLT, the first factor converges weakly to a standard normal random variable, while the second factor is an indicator of a set that has probability that converges to $0$, so the expression converges to $0$.

\subsection{Shrinkage Estimators for Testing Mediation} \label{sec:rates-teststats}
\label{sec:filt-teststats}

Consider the multiple mediation problem, with a collection of hypotheses \begin{equation}\label{eq:medhypmulti}
\hypothesis[\hypname{0}[i]]{\regcoef{1}[i] \regcoef{2}[i] = 0}[\hypname{1}[i]]{\regcoef{1}[i] \regcoef{2}[i] \neq 0} \text{,\quad for}\ i =  1,\dots,m \fullstop
\end{equation}
For each of the hypotheses, the null space contains parameter points with at least one coordinate equals to \(0\), and similarly composed to three sub-spaces as in \eqref{eq:threecases}. In our setting, the aim of the filtration step is to identify null hypotheses of Case~\nullcase{00}, based on the different properties that these null hypotheses have. The filtration test thus, has a simple null hypothesis form \[\hypothesis[\filhyp{0}[i]]{\theta^i = \theta_0}[\filhyp{1}[i]]{\theta^i \neq \theta_0}\comma\] where \(\medparam[i] = \medparamfull[i]\) and \(\medparamzero = \left(0,0 \right) \).

One approach for finding test statistic is using the likelihood-ratio test,
\begin{equation} \label{eq:lrt}
	\Lambda = \frac{\sup_{\theta \in \Theta}{\likelihood{\theta}{X}}}{\sup_{\theta \in \Theta_{00}}{\likelihood{\theta}{X}}} \fullstop
\end{equation}
Direct calculation shows that the likelihood-ratio test statistic is \(\left\|\left(  \estim{1}, \estim{2} \right) \right\|^2 \), so it can be considered as a natural test statistic \(T_\filter\) for the filtration test.

The likelihood-ratio approach can be applied also to the base test statistic \(T\). Here, the separation is between the null \(\Theta_0 = \left( \left\lbrace 0 \right\rbrace \times \real \right) \cup \left( \real \times \left\lbrace 0 \right\rbrace \right) \) and \(\Theta = \real^2\).  Given the estimator \(\estim{1}\) and \(\estim{2}\), the numerator always equals to a constant, i.e.\ the value the likelihood function returns for \(\theta = \left( \estim{1}, \estim{2} \right) \). The denominator is the likelihood function's value when \(\theta = \left(\estim{1}, 0 \right)\) if \(\estim{1} \le \estim{2}\) and \( \theta = \left(0, \estim{2} \right)\) otherwise. Therefore, the test statistic is \(\minestim\), which is equivalent to \(\maxregpval\). This supplies a justification for choosing this statistic for the base test hypothesis.

Phrasing the abstract separation idea can be as follows. We would like to find a function \(\phi\) which maximizes \(P_{\theta }\left( \phi\left( \estim{1}, \estim{2} \right) > t \right)\) under an arbitrary parameter \(\theta \), which is not in Case~\nullcase{00}, while maintaining the same probability \(P_{\theta }\left( \phi\left( \estim{1}, \estim{2} \right) > t \right) \leq \alpha\) for \(\theta \) in Case~\nullcase{00}. Not always there is a perfect test which has the uniformly best power among those with the same level. 

Let us restrict ourselves only to filtration tests \(\nfilevent\) of the form \(\phi\left(\estim{1},\estim{2} \right) \leq t \). Thus, we should consider and compare various functions \(\phi\), when the purpose of each of them is to \emph{separate} or \emph{distinguish} between the point \(\left(0,0 \right) \) and other points in \(\Theta\).
Later on we will study the function \(\phi:\left(x, y \right) \mapsto xy\) in depth. We will also discuss other functions, such as \(\left(x,y\right) \mapsto \left\|\left(x,y\right)\right\|\), \(\left(x,y\right) \mapsto \left\|\left(x,y\right)\right\|^2\), \(\left(x,y\right) \mapsto  \max{\left\lbrace \left| x\right| ,\left| y\right| \right\rbrace }\), \(\left(x,y\right) \mapsto \left|x\right|  + \left|y\right|\), and in general \(\left(x, y \right) \mapsto \left\|\left(x,y \right)  \right\|_p  \), for the \lpspace-norm \(\left\|\cdot \right\|_p \), \(p \in \left[1,\infty\right] \). %Comparison of these functions, theoretically and through simulations, will be given later.
The induced shrinkage estimators for the product estimator would have the form \(\product \indc{\product \in R_\filter}\), and the shrinkage estimator for the screen-min procedure \citep{djordjilovic_optimal_2020}, which filters according to the minimal \pvalue{}, will be \(\minestim \indc {\maxestim \in \filtrr}\). Shrinkage estimators for other statistics can be constructed similarly.

In order to draw useful insights about these estimators as well as others, we need to study and compare their local asymptotic behavior in neighborhoods of parameter points.

%=================================================
%   _____           _   _               _  _   
%   / ____|         | | (_)             | || |  
%  | (___   ___  ___| |_ _  ___  _ __   | || |_ 
%   \___ \ / _ \/ __| __| |/ _ \| '_ \  |__   _|
%   ____) |  __/ (__| |_| | (_) | | | |    | |  
%  |_____/ \___|\___|\__|_|\___/|_| |_|    |_|  
%=================================================
\section{Convergence Rates} \label{sec:rates}

\subsection{Convergence Rates} \label{sec:rates-general} \label{sec:rates-framework}

Let \(\regcoef{1} = \regcoef{1}[A]\) and \(\regcoef{2} = \regcoef{2}[M]\) be the coefficients from the regression equation \eqref{eq:medregmodel}, and let \(\gamma,\beta\) their respective estimators. %Our goal is to test the null hypothesis $\regcoef{1} \regcoef{2} = 0$. This null hypothesis corresponds to the non-existence of a mediator, and can be divided into the three cases in \eqref{eq:threecases}.
%As mentioned before, since the quantity we are interested in, as derived from the regression equations of the mediation hypothesis is \(\regcoef{1} \regcoef{2}\), a natural test statistic for the problem would be \(\product\).
Under standard regularity conditions, \(\sqrt{n}\left( \mle{\regcoef{1}} - \regcoef{1}\right) \to \nordist[0][\sigma^2_{\regcoef{1}}] \) and \(\sqrt{n}\left( \mle{\regcoef{2}} - \regcoef{2}\right) \to \nordist[0][\sigma^2_{\regcoef{2}}] \). Thus, \(\mle{\regcoef{1}} = \avg{\regcoef{1}} + \probasymleq{1}\), \(\estim{1} \sim \nordist[\regcoef{1}][\frac{1}{n} \sigma^2_{\regcoef{1}}]\) and \(\mle{\regcoef{2}}  = \avg{\regcoef{2}} + \probasymleq{1}\), \(\estim{2} \sim \nordist[\regcoef{2}][\frac{1}{n} \sigma^2_{\regcoef{2}}]\), where \(\avg{\regcoef{1}}\) and \(\avg{\regcoef{2}}\) can be thought of as sample means approximations. In this section we will analyze the estimators \(\estim{1}\) and \(\estim{2}\) based on their normal approximation \(\regcoef{1}\) and \(\regcoef{2}\) \citep[see also discussion in][]{huang_genome-wide_2019}.
As we show below, the estimator \(\product\), similarly to the original estimator \(\mle{\gamma} \mle{\beta}\), behaves differently in the parameter point \(\left(0,0 \right) \) compared to the other parameter points in the null.

%\subsection{About Shrinkage Estimators and Others}

Consider the estimator $\widehat{\theta} = \product$.
In order to calculate the exact convergence rate, write
\begin{equation} \label{eq:productdecomp}
	\estim{1} \estim{2} - \regcoef{1} \regcoef{2} = 
	\left( \estim{1} - \regcoef{1} \right) \left( \estim{2} - \regcoef{2} \right) + \regcoef{2} \left( \estim{1} - \regcoef{1} \right) + \regcoef{1} \left( \estim{2} - \regcoef{2} \right) \fullstop
\end{equation}
%When \(\regcoef{1} \ne 0 \ortext \regcoef{2} \ne 0\), the estimator converges to the point in the standard rate of root-\({n}\). Indeed, by the CLT, the first summand  has \(n\)-rate of convergence. but at least one of the two others summands has rate of root-\(n\), so the total rate is root-\(n\).
When \(\regcoef{1} = \regcoef{2} = 0\), the last two summands vanish, and we are left with one summand of \(n\)-rate. 

Another way of showing that is using the delta method. \citet{sobel_asymptotic_1982} showed that, \[\sqrt{n}\left(\left(\estim{1}, \estim{2}  \right) - \left(  \regcoef{1}, \regcoef{2} \right) \right) \to \nordist[0][\Sigma] \] for \(\Sigma = \operatorname{diag}\left( \sigma_{\regcoef{1}}^2, \sigma_{\regcoef{2}}^2 \right) \). Then, for \(\phi\left(x, y \right) = xy \),
\begin{equation} \label{eq:deltaproduct}
	\begin{gathered}
	\sqrt{n}\left( \product - \regcoef{1} \regcoef{2} \right) = \sqrt{n}\left(f\left(\estim{1}, \estim{2}  \right) - f\left(  \regcoef{1}, \regcoef{2} \right) \right) \to \\ \nordist[0][\transpose{f'\left(\regcoef{1}, \regcoef{2} \right)} \Sigma f'\left(\regcoef{1}, \regcoef{2} \right)] = \nordist[0][\regcoef{1}^2 \sigma_{\regcoef{2}}^2 + \regcoef{2}^2 \sigma_{\regcoef{1}}^2] \fullstop
	\end{gathered}
\end{equation}
Note that for parameter points in Cases \nullcase{01} and \nullcase{10}, the convergence rate is root-\(n\). As already proved in exact terms in \eqref{eq:productdecomp}, Case~\nullcase{00} yields a higher rate.

	The delta method argument can be generalized to show that the estimator \(\product\) is not the only example for estimator with higher rate of convergence in \(\left(0,0 \right) \).

Consider the shrinkage estimator in \eqref{eq:filtestimator} where the filtration test statistic \(\filtstat\) has the form \(\phi\left( \estim{1}, \estim{2} \right) \). Interestingly, under some conditions on \(\phi\), the convergence rate of \(\phi\left( \estim{1}, \estim{2} \right) \) might depend on the parameter point.

\begin{lemma} \label{lem:higherrate}
	Let \(\phi:\real^p\rightarrow\real\) be a function which is differentiable and attains its minimum at \(\theta_0\). Let \(T_n = \left( T_{n1},\dots,T_{np} \right) \) be a consistent estimator for \(\theta\) with convergence rate \(r_n\). Then, the convergence rate of \(\phi\left( T_n \right)\) at \(\theta_0\) that is faster than \(r_n\).
\end{lemma}

\begin{proof}
	Since \(\phi\) is convex and symmetric, it attains a global minimum at \(\theta_0\). Let \(r_n\left( X_n - \theta \right) \convdist L\), for some distribution law \(L\). By the delta method \citep[Theorem~3.1]{van_der_vaart_asymptotic_1998}, \(r_n\left( \phi\left( X_n\right)  - \phi\left( \theta \right)\right) \convdist \phi'\left(\theta_0 \right) \cdot L \).
	At the minimum point \(\theta_0\) the derivative vanishes, so the latter limit equals 0, and thus the rate is higher than \(r_n\).
\end{proof}

 Consider for instance the separation function \(\phi\left(x,y \right) = x^2 + y^2 \). This function can be used for creating a test statistic for the filtration test. Using Lemma~\ref{lem:higherrate}, it can shown that \( \estim{1}^2 + \estim{2}^2\) is another example for an estimator with higher rate of convergence at \(\left(0,0 \right) \).

\subsection{Irregularity of Estimators} \label{sec:rates-irreg}

One way of coping with the problem of different convergence rates is using normalization. Consider the following two examples: the squared-norm estimator \( \estim{1}^2 + \estim{2}^2\), and the product estimator \(\product\). As mentioned in the previous section, for both of them, the convergence rate at 0 is different from other parameter point, which can be thought of as a singularity point of the standard deviation in \(0\). Dividing these estimators by (an estimator for) their limit standard deviation, as it appears in the limit distribution in the delta method, can be used in order to remedy the difference in rates.
%For the squared-norm estimator the variance is \(\sigma_{\regcoef{2}}^2 \regcoef{1}^2 +  \sigma_{\regcoef{1}}^2 \regcoef{2}^2 + \sigma_{\regcoef{1}}^4 + \sigma_{\regcoef{2}}^4\), and for the product estimator the variance is
%\(\sigma_{\regcoef{1}}^2 \regcoef{1}^2 +  \sigma_{\regcoef{2}}^2 \regcoef{2}^2 + \sigma_{\regcoef{1}}^2 \sigma_{\regcoef{2}}^2\). In these expressions, in case that the variance is unknown, it can be estimated, and the 4th-power terms of \(\sigma_{\regcoef{1}}\) and \(\sigma_{\regcoef{2}}\) are neglected since they are disapper in the limit.
The obtained normalized estimators, up to an asymptotically negligible terms in the standard deviation are, respectively, \[\frac{ \estim{1}^2 + \estim{2}^2}{\sqrt{\sigma_{\regcoef{1}}^2 \regcoef{1}^2 +  \sigma_{\regcoef{2}}^2 \regcoef{2}^2}} \comma\]
and
%consider the normalized version \(\sqrt{\estim{1}^2 + \estim{2}^2}\). For the product estimator \(\product\), consider the limit variance as it appears in \eqref{eq:deltaproduct}, and let us estimate this quantity by \(\sqrt{\sigma_{\regcoef{2}}^2 \estim{1}^2 +  \sigma_{\regcoef{1}}^2 \estim{2}^2}\). The normalized estimator is
the so-called Sobel's test statistic \citep{sobel_asymptotic_1982, mackinnon_comparison_2002, mackinnon_introduction_2012, huang_genome-wide_2019},
\begin{equation} \label{eq:sobel}
\frac{\product}{\sqrt{\sigma_{\regcoef{2}}^2 \estim{1}^2 + \sigma_{\regcoef{1}}^2 \estim{2}^2}} \fullstop
\end{equation}
Sobel's test is usually performed as a \ttest\ or \ztest\, based on the normal approximation in~\eqref{eq:deltaproduct}.

%This approximation is problematic in Case~\nullcase{00}, as mentioned in \citet{sobel_asymptotic_1982}, which stems from similar arguments concerning the different rate in \(\left(0,0 \right) \).
%The purpose of the normalizations offered is to reduce side effects caused by the change in convergence rate.
It can be shown that the convergence rate of both estimators is indeed root-\(n\), even at Case~\nullcase{00}. However, the asymptotic behavior of these estimators is still problemtic for Case~\nullcase{00}, and for the sake of formulating this idea, we will first need the following definition of regular estimators.

\begin{definition}[{\citealp[Chapter~8.5]{van_der_vaart_asymptotic_1998}}] \label{def:regest}
	An estimator \(T\) for \(\psi \left( \theta \right) \) is regular at a parameter point \(\theta\) if
	\[\sqrt{n} \left( T - \psi \left( \theta + \frac{h}{\sqrt{n}} \right) \right) = L_\theta \]
	\(L_\theta\) is the same limit distribution for every \(h \in \real\).
\end{definition}

The following lemma, whose its proof is given in the Appendix, sheds light on the behavior of the normalized estimators in \(\left( 0,0\right) \).

\begin{lemma} \label{lem:irregularity}
	Both of the estimators \[\frac{\product}{\sqrt{\sigma_{\regcoef{2}}^2 \estim{1}^2 + \sigma_{\regcoef{1}}^2 \estim{2}^2}}\] and \[\sqrt{\sigma_{\regcoef{2}}^2 \estim{1}^2 + \sigma_{\regcoef{1}}^2 \estim{2}^2}\] are  irregular at \(\left(0,0 \right) \).
\end{lemma}

Note that the original, unnormalized estimators, \(\product\) and \(\sigma_{\regcoef{2}}^2 \estim{1}^2 + \sigma_{\regcoef{1}}^2 \estim{2}^2\), are regular at \(\left(0,0 \right) \) as the limit distribution is always \(0\). As it turns out, by performing the normalization, the problem of different rates of convergence is only replaced with the irregularity at the point \(\left(0,0 \right) \), which makes the estimator's limit behavior depend on small neighborhood of the parameter and the convergence direction of the parameter sequence.

As a conclusion of the previous discussion, we see that different estimators, have various strengths and weaknesses. Therefore, choosing the base and filtration test statistics for the shrinkage estimator in \eqref{eq:filtestimator} becomes complicated. There is not necessarily one optimal way to do so, and there is a trade-off to be considered between the different options.

% =============================================
%   _____           _   _               _____ 
%   / ____|         | | (_)             | ____|
%  | (___   ___  ___| |_ _  ___  _ __   | |__  
%   \___ \ / _ \/ __| __| |/ _ \| '_ \  |___ \ 
%   ____) |  __/ (__| |_| | (_) | | | |  ___) |
%  |_____/ \___|\___|\__|_|\___/|_| |_| |____/ 
% =============================================

\section{Local Asymptotic Analysis} \label{sec:la}

\label{sec:la-localparam}

In order to examine efficiency of estimators in the neighborhood of a parameter point, we use the local asymptotics framework, which basically lets the parameter space depend on $n$. In this section, we follow \citet[Chapters 7, 9, and 15]{van_der_vaart_asymptotic_1998}. Formally, the probability measure is of the form \(P_{n,\theta} = P_{\theta_n}^n\), which means that as the number of observations increases, the true parameter also changes. Usually \(\theta_n = \theta + r_n^{-1}h\), i.e., \(\theta_n\) converges to a constant parameter \(\theta\) with rate \(r_n\).

%	For every $n \in \mathbb{N}$, let $\left( \Omega_n , \mathcal{F}_n \right)$ be a measurable space, and let the parameter space $\Theta\subseteq\mathbb{R}^d$ be an open subset. Let $\mathcal{P}_n = \left\lbrace P_{n,\theta} \,\middle|\, \theta \in \Theta \right\rbrace$ be a statistical model over $\Theta$, i.e.,~$P_{n,\theta}$ is a probability measure for every $n \in \mathbb{N}$ and $\theta \in \Theta$. An estimator is a measurable function $T:\Omega_n \to \Theta$. In a less general case, one can assume $\Omega_n = \prod_{i=1}^n \Omega$ for a fixed $\Omega$, and for every $\theta\in\Theta$, $P_{n,\theta}= \prod_{i=1}^n P_\theta$ for a fixed $P_\theta$, a probability measure over $\Omega$, which corresponds to $n$ \iid\ observations over $\Omega$ with probability $P_\theta$. We will denote $\mean{T} = \int_{\Omega_n}{T} dP_{n,\theta}$.

Back to the Hodges example presented in \eqref{eq:hodges}, the Hodges' estimator should also be examined in the local asymptotic framework, for a sequence of parameters \(\mu_n\) that approaches~\(0\). When \(\mu_n = \asymleq{n^{-1/4}}\) and also \(\mu_n = \asymgtr{n^{-1/2}}\), for instance when \(\mu_n = n^{-1/3}\), then \(n^{1/2}\left(\widetilde{\mu} - \mu \right) \to -\infty \) \citep[Example~8.1]{van_der_vaart_asymptotic_1998}. This means that although the estimator looks more efficient than the regular one at \(0\), there are sequences of parameters \(\mu_n\) which their normalized distance from the estimator and diverges. The exact rate of convergence depends on the problem in question.

%When $\mu_n = n^{-\gamma}$ for $\gamma > \frac{1}{4}$, then as $n\to\infty$,
%\begin{equation*}
%\prob[][]{\left| n^{1/2}\,\avg{X} \right| >n^{1/2}n^{-1/4} }
%\approx \prob[][]{\left| Z + n^{1/2}\mu_n \right| > n^{1/4}} \rightarrow 0 \,.
%\end{equation*}
%Therefore, for $\frac{1}{4} < \gamma < \frac{1}{2}$,
%\begin{equation*}
%n^{1/2}\left( \widetilde{\mu}_n - \mu_n \right) =
%\underbrace{n^{1/2}\left( \avg{X} - \mu_n \right) }_{\to \nordist}
%\underbrace{\indc{\left| \avg{X} \right| > n^{-1/4} }}_{\to 0} -
%\underbrace{n^{1/2}\mu_n}_{\to\infty}
%\underbrace{\indc{\left| \avg{X} \right| \le n^{-1/4} }}_{\to 1} 
%\to -\infty \,.
%\end{equation*}

%The central limit theorem states that, for a sequence of \(n\) \iid\ variables \(X_1,\dots,X_n\), \[\sqrt{n}\left(\avg{X} - \mu \right) \to \nordist[][\sigma^2] \], for \(\mu\) and \(\sigma\) the mean and the variance of the underlying distribution of the \(X_i\), respectively. In words, the average, as an estimator for the mean, is consistent, with root-\(n\)-rate of convergence. It means that estimating the mean with the average, or testing hypothesis such as \[\hypfull{\mu = \mu_0}{\mu \neq \mu_0}\], are trivial tasks, asymptotically. In order to examine anc compare asymptotic efficiency of estimators and tests, the alternative parameter points have to be in a neighborhood of the true parameter.

\subsection{Characterization of Shrinkage Estimators} \label{sec:la-char}
Filtration estimators, in general, have properties similar to the Hodges' estimator. Asymptotic improvement in the rate of convergence in some points exists only with corresponding low performance in others. We aim to identify the regions in which a filtration estimator is better as well as weaker than the original one.

	The common measure for comparing efficiency of two estimators \(T_1\) and \(T_2\) of a parameter \(\psi\left(\theta \right) \) in a parameter point \(\theta\) is the \emph{asymptotic relative efficiency (ARE)}. The ARE is equivalent to the limit ratio of the corresponding variances,
	\[\operatorname{ARE}\left(T_1, T_2 ; \theta \right) = \lim \frac{\sigma_2^2\left(\theta\right)}{\sigma_1^2\left(\theta\right)} \fullstop\]
	The interpretation for the ARE is the rate it takes to \(T_2\) to estimate \(\psi\left(\theta \right) \) with variance~\(1\), relative to \(T_1\). Thus, ARE less than~\(1\) indicates that \(T_2\) is more efficient than \(T_1\), and vice versa.
	
	The ARE, as formulated explicitly above, takes into account only the variance of the estimators, and not the asymptotic bias rate of convergence. In Hodges-like filtration estimators, the asymptotic bias is usually the cost for the super-efficiency in some points. In order to consider both these quantities, the asymptotic variance and the asymptotic bias, into the comparison, we shall compare their mean squared error (MSE). Indeed, the MSE weighs the variance and the square of the bias with equal weights: \(\MSE\left(T,\theta \right) = \mean{\left(T-\theta \right)^2 } = \var{T} + \mean{T-\theta}^2 \). Generalizing the definition of ARE, we treat MSE-ratio less than \(1\) as indicating that \(T_2\) is more efficient than \(T_1\) and vice versa.

	Let \(M\) be the limit MSE-ratio of \(T_2\) relatively to \(T_1\). We say that \(T_2\) is much more efficient than \(T_1\) when \(M = 0\); more efficient when \(0<M<1 \), equivalent when \(M = 1\); less efficient when \(1<M<\infty \); and much less efficient when \(M = \infty\).

	\begin{assumption} \label{assum:asymunbiased}
		The original estimator \(\estimreg\) is \(r_n\)-consistent for some rate \(r_n\), such that the sequence \(r_n\left(T-\psi\left(\theta \right)  \right) \) has limit distribution with mean \(0\) and finite variance.
		% and finite fourth moment. 
		%that is \(\estimreg = \psi\left( \theta_n \right) + O_P\left( r_n^{-1}\right)  \). Moreover, we assume that
		%		\begin{equation} \label{eq:asymunbiased}
		%		\mean{r_n\left(\estimreg - \psi\left( \theta_n\right)  \right) } \to 0 \fullstop
		%		\end{equation}
		%
		
		%[[? Moreover, assume that the sequence \(T_n\) is uniformly integrable,\\and that the limit distribution is positive over \(\real\),\\and that the indicator converges in probability to a limit indicator with mean \(L\) ?]]
	\end{assumption}
	This assumption holds true, for example, for all asymptotically linear estimators, with \(r_n = \sqrt{n}\). Under this assumption, the following theorem holds.
	
%	\begin{assumption} \label{assum:poslimint}
%		The limit expectation \(\mean{r_n\left( \left(\estimreg - \psi\left(\theta_n\right)\right)  \right)^2 \indc*{A} } \) is positive for every set \(A\) with \(P\left(A \right) >0\).
%	\end{assumption}
	
%	\begin{proposition}
%		If, for a sequence \(\estimreg\), every random variable \(\estimreg\) has CDF which is continuous and the series of CDF uniformly converges to the limit distribution, then Assumption~\ref{assum:poslimint} holds. (???)
%	\end{proposition}

	%The finite fourth moment can be relaxed by assuming uniform integrability and by further complicating the proof. (???)

	\begin{theorem} \label{thm:effclass}
		Let \(\basestat[][n]\) be the regular estimator, and \(\shrinkestim[][n] = \shrinkestimdef[][n] \) be the filtration estimator based on it.
		Let \(L = \lim P_{\theta_n} \left( \filevent[][n] \right) \) and \(K = \lim \frac{1}{\sigma_{T_n}} \left( \psi\left(\theta_n\right) - \psi\left(\theta_0\right) \right)\). %(or, equivalently, \(\lim P\left( T_n \in R_n\right) \)).
		\begin{enumerate}[label = \roman*.]
			\item When \(L = 1\), then \(\shrinkestim[][n]\) is much more efficient than \(\estimreg\) when \(K = 0\); more efficient when \(0 < K < 1\); equivalent when \(K = 1\); less efficient when \(1<K<\infty\); and much less efficient when \(K = \infty\).
			\item When \(0 < L < 1\), then when \(K = 0\), \(\estimfil\) is more efficient than \(\estimreg\), and when \(K = \infty\), the estimator is much less efficient, and for \(0<K<\infty\), the reltaive MSE of the estimators can be either more efficient, equivalent or less efficient.
			\item When \(L = 0\), then \(\estimfil\) equivalent to the original \(\estimreg\).
			
		\end{enumerate}
	\end{theorem}

%}

%		\begin{assumption}
%			The conditional distribution \(\) 
%		\end{assumption}

%		\begin{example}
%			For \(T = \min\left\lbrace \left| \regcoef{1} \right|, \left| \regcoef{2} \right| \right\rbrace \) and \(T_\filter = \max\left\lbrace \left| \regcoef{1} \right|, \left| \regcoef{2} \right| \right\rbrace\) (variation of screen-min procedure), for \(t\)
%			\begin{align*}
%			P\left(T<t \andtext T_\filter > c \right) = 
%			\end{align*}
%		\end{example}

Note that the existence of the cases discussed above is affected by the chosen test statistics and the rejection regions for the base and filtration tests. %Despite that, improvement can only exist with corresponding low performance.

\subsection{Case Study: The Product Statistic} \label{sec:la-casestudy}

Recall that \(\estim{1}\) and \(\estim{2}\) are the estimators for the regression coefficients \(\regcoef{1}\) and \(\regcoef{2}\) in the regression Model \eqref{eq:medregmodel}, respectively. In this section we use Theorem~\ref{thm:effclass} to study the properties of the shrinkage estimator \(\estimfil = \product \indc{n^\delta \left| \product \right| > c}\) based on the base and filter test statistic \(\product = \psi\left(\estim{1},\estim{2} \right) \), for \(\psi\left(x, y \right) = xy \). Its aim, as detailed in Section \ref{sec:filt}, is to filter hypotheses which are likely to belong to the null parameter point \(\left(0,0 \right) \). %The shrinkage estimator will be considered as a case study for the theorem. We define, for \(\estimreg = \product\), the estimator ,
%Written below are results for an example case, of base estimator \(T = \avg{X}\avg{Y}\) and filtration set \(A = \left\lbrace \avg{X}\avg{Y} > c n^{-\gamma} \right\rbrace \). These results are shown here for the sake of demonstrating and characterizing the filtration process in a case study.

%``We consider parameters \(\theta + h/\sqrt{n}\) for \(\theta\) fixed and \(h\) ranging over \(\real^k\) and suppose that, for certain limit distributions \(L_{\theta,h}\), \[\sqrt{n} \left( T - \psi \left( \theta + \frac{h}{\sqrt{n}} \right) \right) = L_{\theta,h} \] 
%for certain limit distributions \(L_{\theta,h}\) Then Tn can be considered a good estimator for \(\psi\left(\theta_n\right) \) if the limit distributions \(L_{\theta,h}\) are
%maximally concentrated near zero. If they are maximally concentrated for every h and
%some fixed \(\theta\), then \(T_n\) can be considered locally optimal at \(\theta\)."

%{\color{blue}\makeatletter\let\default@color\current@color\makeatother

We now consider sequences of parameters \(\theta_n = \regcoef{1}[][n]\regcoef{2}[][n]\) approaching the constant true parameter \(\theta_0 = \regcoef{1}[][0]\regcoef{2}[][0]\). The rate of convergence, \(K = \lim \frac{1}{\sigma_{T_n}} \left( \psi\left(\theta_n\right) - \psi\left(\theta_0\right) \right)\), is in this case \[K = \lim \frac{\regcoef{1}[][n]\regcoef{2}[][n]}{\sqrt{n^{-1}\left(n^{-1}+\regcoef{1}[][n]^2 + \regcoef{2}[][n]^2 \right) }} \comma\]

The value of \(K\) can be partially characterized as a function of \(\regcoef{1}[][n]\) and \(\regcoef{2}[][n]\). Without loss of generality, we assume that both \(\regcoef{1}[][n]\) and \(\regcoef{2}[][n]\) are non-negative. Since \(2\regcoef{1}[][n]\regcoef{2}[][n] \leq \regcoef{1}[][n]^2 + \regcoef{2}[][n]^2\), we have
\begin{equation} \label{eq:distancebound}
	K \leq \lim \frac{\regcoef{1}[][n]^2 + \regcoef{2}[][n]^2}{2 \sqrt{n^{-1}\left(n^{-1}+\regcoef{1}[][n]^2 + \regcoef{2}[][n]^2 \right) }} = \lim \frac{n\left( \regcoef{1}[][n]^2 + \regcoef{2}[][n]^2\right) }{2 \sqrt{1 + n \left(\regcoef{1}[][n]^2 + \regcoef{2}[][n]^2 \right) }}\fullstop
\end{equation}
When \(n\left( \regcoef{1}[][n]^2 + \regcoef{2}[][n]^2\right) \to 0\), then \(K = 0\). When \(n\left( \regcoef{1}[][n]^2 + \regcoef{2}[][n]^2\right) \to c\) then \(K \leq \frac{c}{\sqrt{1+c}}\), so for \(c\) smaller than the golden ratio \(\phi = \frac{1+\sqrt{5}}{2}\), then, by \eqref{eq:distancebound}, we have \(K < 1\). Similarly, it can be bounded below: When \(n\left( \regcoef{1}[][n]^2 + \regcoef{2}[][n]^2\right) \to \infty\), then \(K \geq \infty\), and when \(n\left( \regcoef{1}[][n]^2 + \regcoef{2}[][n]^2\right) \to c\) then \(K \geq \frac{c}{\sqrt{1+c}}\).

The limit probability \(L = \lim P_{\theta_n} \left( \filevent \right)\) can be characterized into regions according to Table~\ref{tbl:prodclassprob}, where
\(A = \max \left\lbrace
\lim n^{\delta}\regcoef{1}[][n]\regcoef{2}[][n],
\lim n^{\delta - 1/2} \sqrt{n^{-1} + \regcoef{1}[][n]^2 + \regcoef{2}[][n]^2}
\right\rbrace \):
\begin{table}[h]
	\centering
	\begin{tabular}{|c||c|c|c|}
		\hline
		& \(A = 0\) & \(0 < A < \infty\) & \(A = \infty\)\\
		\hline \hline
		\(\delta < 1\) & \(L = 1\) & \(0 < L < 1\) & \(L = 0\) \\
		\hline
		\(\delta = 1\) & --- & \(0 < L < 1\) & \(L = 0\)\\
		\hline
		\(\delta > 1\) & --- & --- & \multicolumn{1}{|c|}{\(L = 0\)} \\
		\hline
	\end{tabular}
	\caption{Classification into regions}
	\label{tbl:prodclassprob}
\end{table}

The classification in Table~\ref{tbl:prodclassprob} comes from approximating the statistic by its mean and standard deviation, 
\begin{equation} \label{eq:prodapprox}
	\product = \mean{\product} + \asymstocleq \left( \var{\product}^{1/2} \right) = \regcoef{1}[][n]\regcoef{2}[][n] + \asymstocleq \left( \sqrt{n^{-1}\left( n^{-1} + \regcoef{1}[][n]^2 + \regcoef{2}[][n]^2\right) } \right) \fullstop
\end{equation}

Explanation to the classification in the table may be provided as follows. When \(\delta < 1\), the statistic \(n^{\delta}\product\) might be bounded in high probability by the constant \(c\). When both \(\lim n^{\delta}\regcoef{1}[][n]\regcoef{2}[][n] = 0\) and \( \lim n^{\delta - 1/2}\sqrt{\regcoef{1}[][n]^2 + \regcoef{2}[][n]^2} = 0\), then \(\product \to 0\), which means \(L = 1\). The rate of convergence \(K\) directly influences the MSE-ratio.

When \(\delta > 1\), we have that \(L = 0\), for every rate of convergence to the parameters. Indeed, the estimator \(n^{\delta}\product\) will have to have standard deviation which converges to \(\infty\), so it is not possible for the estimator to be bounded in positive probability. When \(A = \infty\), for any value of \(\delta\), the same reasoning holds, since by definition either the mean or the standard deviation converges to \(\infty\) and cannot be bounded in positive probability, which means that the estimators are equivalent for every rate of convergence of the parameters.

When \(0<A<\infty\), then \(0<L<1\). It may happen for \(\delta = 1\), and then there is no possibility of the extreme cases \(K = 0\) and \(K = \infty\); When \(n^{1/2}\regcoef{1}[][n], n^{1/2}\regcoef{2}[][n]\to h \), \(0<h<\infty\), \(K\) also gets values in \(\left(0,\infty \right) \), and so as the MSE-ratio. When \(h = 0\), the MSE-ratio is also in \(\left(0,\infty \right) \), and \(h = \infty\) is not possible since it belongs to the case \(L = 0\).

Simulations of different MSE-ratios, as functions of \(L\) and \(K\), are shown in Figure~5.1. The filtration test in Figure (a) is \(\left| \product\right| <4n^{-0.7}\), in Figure (b) is \(\left| \product\right| <2.5n^{-1}\), and in Figure (c) is \(\left| \product\right| <2.5n^{-1.5}\). The true means \(\left(\regcoef{1}[][n], \regcoef{2}[][n] \right) \) for Figure (a) are (from the least value of \(K\) to the greatest) \(\left( n^{-0.5}, n^{-0.6}\right) \), \(\left( n^{-0.5}, n^{-0.5}\right) \), \(\left( 2n^{-0.5}, \sqrt{\frac{5}{3}}n^{-0.5}\right) \), \(\left( 2n^{-0.5}, 2n^{-0.5}\right) \) and \(\left( 2n^{-0.4}, n^{-0.4}\right) \). The true means for Figure (b) are (from the least value of \(K\) to the greatest) \(\left( 0.7n^{-0.5}, 0.7n^{-0.5}\right) \), \(\left( 1.075n^{-0.5}, 1.075n^{-0.5}\right) \),  and \(\left( 2n^{-0.5}, 2n^{-0.5}\right) \). The true means for Figure (c) are \(\left( n^{-1}, n^{-1}\right) \).

The intermediate case can be bounded theoretically, as shown in the Appendix.

%}

\begin{figure}
	\centering
	\begin{subfigure}{0.3\textwidth}
		\includegraphics[width=\textwidth]{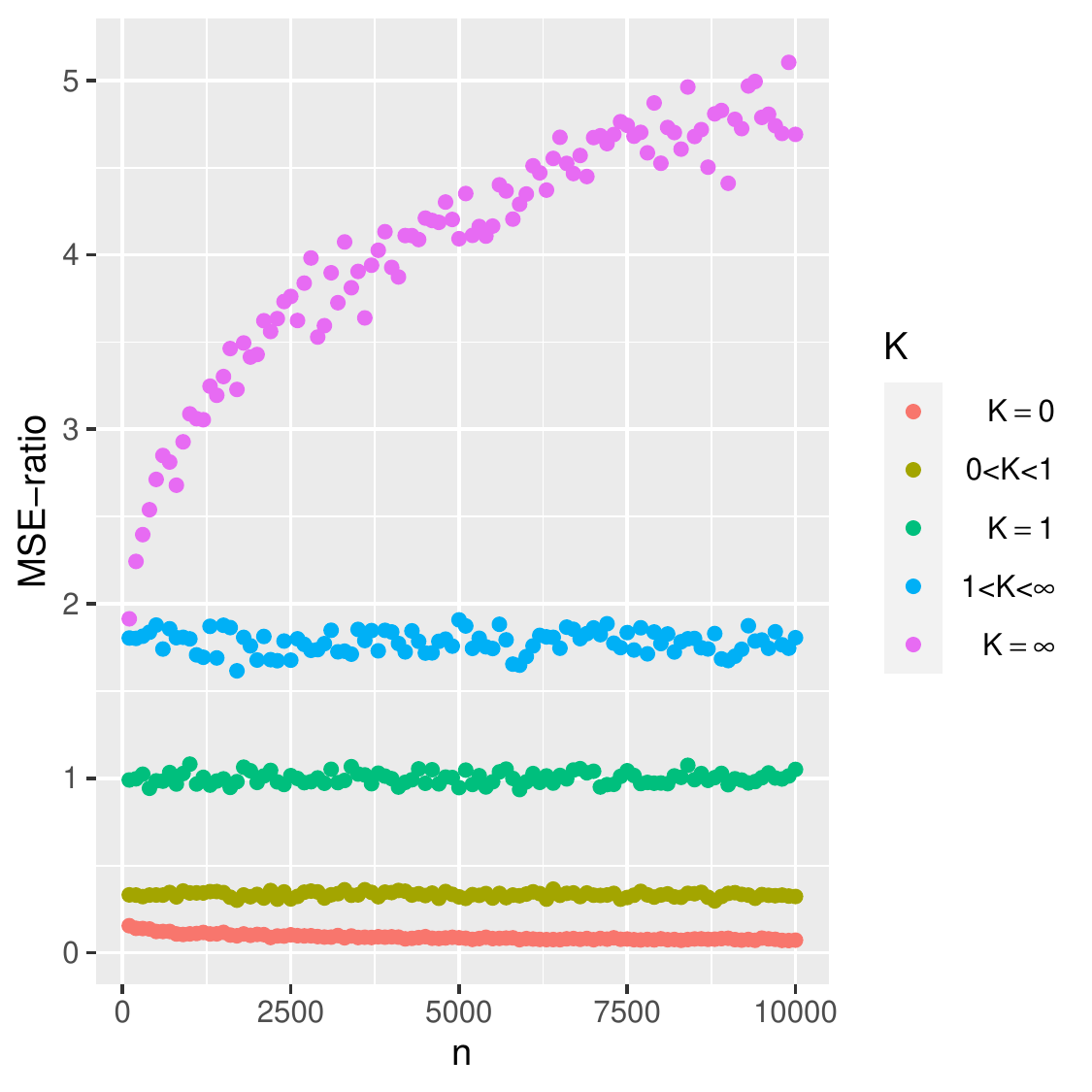}
		\caption{\(L = 1\)}
	\end{subfigure}
	\hfill
	\begin{subfigure}{0.3\textwidth}
		\includegraphics[width=\textwidth]{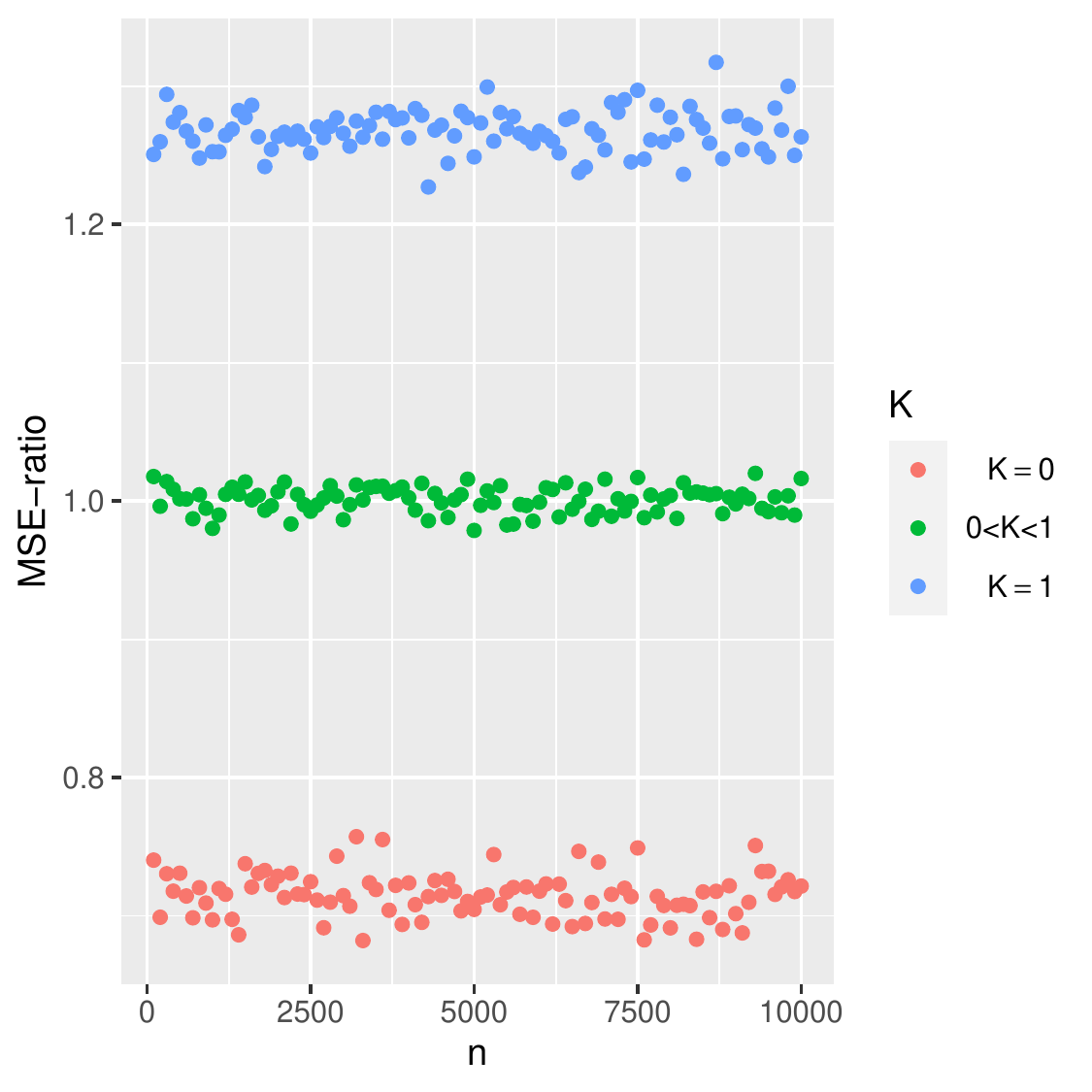}
		\caption{\(0 < L < 1\)}
	\end{subfigure}
	\hfill
	\begin{subfigure}{0.3\textwidth}
		\includegraphics[width=\textwidth]{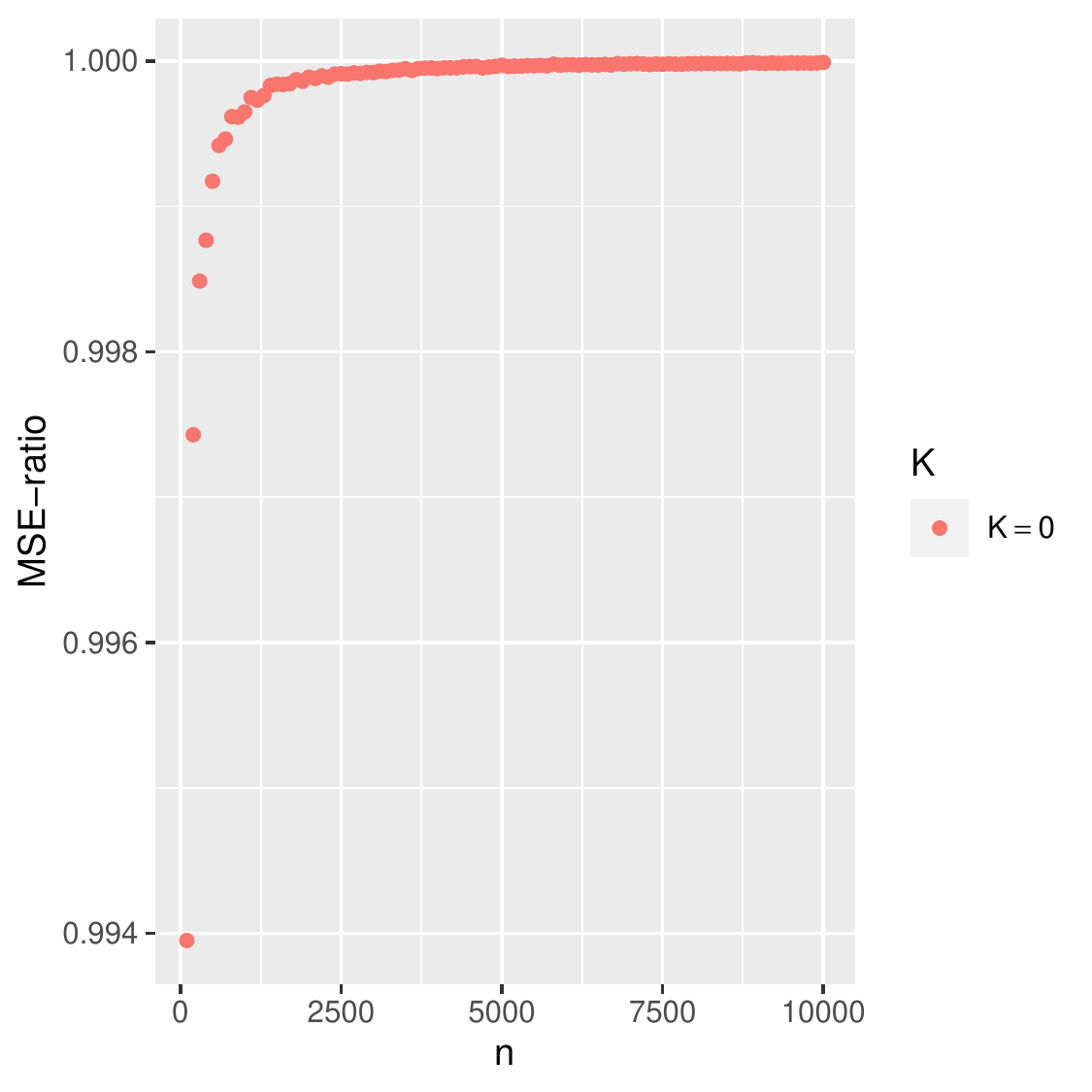}
		\caption{\(L = 0\)}
	\end{subfigure}
	\caption[MSE-ratio illustration]{Simulation of different MSE-ratios, as functions of \(L\) and \(K\).}  \label{fig:casestudycases}
\end{figure}
%\mseratioa{\mseratioakm}{\mseratioabetaexp}

\subsection{Multiple Mediation Hypotheses} \label{sec:la-multmedhyp}

So far, we discussed the asymptotic behavior of shrinkage estimators under different parameter points. Now, we would like to return to the multiple-hypotheses framework. 

Considering alternatives as parameter points approaching to the null with the usual rate of \textrootn, i.e., the alternative has the form \(\theta = \theta_0 + \frac{h}{\sqrt{n}}\), where \(\theta_0\) is in the null. Hence, the parameter of interest is a re-scaled version of the original parameter, which can be rephrased as \(h = \sqrt{n}\left(\theta-\theta_0 \right) \).

Furthermore, the parameter space as a whole is re-scaled and gets the form \(H = \sqrt{n}\left(\Theta - \theta_0 \right) \). The notion of local parameter space, enables studying the asymptotic properties by capturing the local behavior around parameter points. When the parameter point \(\theta_0\) is in the interior of the space \(\Theta \subseteq \real^k\), the local space converges to whole \(\real^k\). Otherwise, the local space structure depends on the geometry of the parameter space at a vicinity of \(\theta_0\). 

When the null hypothesis consists of a unique interior point, the geometry of the local null space is trivial, yet the situation is still complicated. As discussed in \citet[p.~217]{van_der_vaart_asymptotic_1998}, \textit{``if \(\theta\) is of dimension \(k>1\), then there exists no uniformly most powerful test, not even among the unbiased tests. A variety of tests are reasonable, and whether a test is `good' depends on the alternatives at which we desire high power"}.

For composite hypotheses, it can be even more challenging.
There is need to consider the local null parameter spaces as well, namely \(H_0 = \sqrt{n}\left(\Theta_0 - \theta_0 \right) \) for points \(\theta_0 \in \Theta_0\). These local spaces can have different properties such as geometry and rates of convergence. When \(\theta_0\) is an inner point of \(\Theta_0\) in the usual Euclidean topology, then the local null space converges to~\(\real^k\). Otherwise, the geometry can be non-trivial and may affect the properties of the test in the limit, such as the distribution under the null.

When testing for mediation, the null hypothesis is composite, so the local parameter depends on the true point in the null. More formally, for the null point \(\left( 0,0 \right) \), the local space of null parameter points is composed of \(\Set{\left(h, 0\right) | h \in \real} \cup \Set{\left(0, h\right) | h \in \real}\), while the alternative consists \(\Set{\left(h_1, h_2\right) | h_1 h_2 \neq 0}\).
However, the local null space of \(\left(c, 0 \right) \), \(c\neq 0\), is \(\Set{\left(c + h, 0\right) | h \in \real}\), while the local alternative space is \(\Set{\left(c + h_1, h_2\right) | h_2 \neq 0}\).
These spaces are illustrated in Figure \ref{fig:localspaces}.
As a result, these different spaces shed light on the difficulty in estimation and testing for mediation.

\subsection{Controlling the FWER}
As a conclusion from the above discussion, considering null points as divided into the discrete cases does not take into account the local behavior of the points. Using this understanding, we provide two ways to bound the finite-sample FWER which take this understanding into account.

In the following we provide an upper bound on the FWER for the two-stage procedure presented in Section~\ref{sec:filt-alg}.  %We would like to state a result concerning the general case of two-stage procedures under the local-asymptotics framework.
For set of \(m\) hypotheses, we denote by \(f \in \left\lbrace 0, 1 \right\rbrace^m \) the vector whose \(i\)-th coordinate is \(\indc{\nfilevent}\). Let \(F = \sum_{i=1}^{m} f_i\).

\begin{theorem} \label{thm:fwercontrol}
	Assume that all \(m\) mediation hypotheses are independent.
	Assume that \(\theta_0 = \arg\min_{\theta}{P_\theta\left(\nfilevent \right) }\) (the parameter point with the lowest filtration probability is \(\theta_0\)). Assume also we use an adaptive adjusting method (after filtration step, and after computing \(F\), number of unfiltered hypotheses), such that under the null \(P_{\theta}\left(\cond{\baserejeventadj}{F} \right) = \frac{P_{\theta_0}\left( \nfilevent \right)}{F} \alpha \). For the special case that \(T\) is a p-value, it can be achieved for example by multiplying the rejection threshold by \(\frac{P_{\theta_0}\left( \nfilevent \right)}{F}\) (similarly to Bonferroni). Let \(V\) be the number of true null rejected. Then
	\(P\left(V \geq 1 \right) \leq \alpha \).
\end{theorem}

The following result generalizes the FWER bound given by Proposition~1 of \citet{djordjilovic_optimal_2020} for the two-stage screen-min procedure, in which \citet{djordjilovic_optimal_2020} assumed discrete separation of the null into cases (Denoted here as Cases \nullcase{00}, \nullcase{01} and \nullcase{10}).

\begin{theorem} \label{thm:fwerbound}
	Assume that all \(m\) mediation hypotheses are independent.
	Assume we use an adaptive adjusting method (after filtration step, and after computing \(F\), number of unfiltered hypotheses), such that under the null \(P_{\theta}\left(\baserejeventadj \right) = \alpha/F \). For the special case that \(T\) is a p-value, it can be achieved for example by multiplying the rejection threshold by \(1/F\) (similarly to Bonferroni). Let \(V\) be the number of true null rejected. Then
	\[P\left(V \geq 1 \right) \leq \mean{\left(1 - \left( 1 - \max_{\theta\in\Theta }P_\theta \left(\cond{T^i \in R^i}{S^i \in Q^i}\right) \right)^{F}  \right)\indc{F>0} } \fullstop \]
\end{theorem}

As stated, hypotheses with filtration probability \(0\) are not altered by the filtration stage, thus remain with the same (wanted) behavior of the original, base test. Other hypotheses may be filtered, depending on the filtration stage.

The advantage of the first theorem is that it is possible, at least numerically, to calculate \(P_{\theta_0}\left( \nfilevent \right)\), thus the bound can easily be used. While the bound in the second theorem may be less conservative, it is more difficult to explicitly calculate it, as the bound involves both calculating \(\max_{\theta\in\Theta }P_\theta \left(\cond{T^i \in R^i}{S^i \in Q^i}\right)\) and taking expectation. Relative to a fixed alternative, both bounds can be used to optimize the power while controlling the FWER. As discussed in Section 5.3, the challenge is more complicated since the alternatives are unknown. Given a prior over both null and alternative parameter points, one can optimize the tests to achieve better bounds, see, for example, discussion in \citet{djordjilovic_optimal_2020}, Section 4. Specifically, while the stated theorems concern controlling the finite-sample FWER, \(L\) and \(K\) in Theorem~\ref{thm:effclass} with the local asymptotic framework can be used to achieve better bounds for the asymptotic FWER. These are important questions, which deserve future research.

\begin{figure}
	
	\centering
	\begin{tikzpicture}
		\fill[yellow!20!white] (0,0) circle [radius=1cm]
		                       (3,0) circle [radius=1cm];
		\fill[yellow!20!white] (0,3) circle [radius=1cm];
		
		\draw[->, very thick] (-0.75,0) -- (0.75,0);
		\draw[->, very thick] (0,-0.75) -- (0,0.75);
		
		\draw[->, very thick] (2.25,0) -- (3.75,0);
		%\draw[->, very thick] (3,-0.75) -- (3,0.75);
		
		%\draw[->, very thick] (-0.75,3) -- (0.75,3);
		\draw[->, very thick] (0,2.25) -- (0,3.75);
		
		\filldraw (1.25,0) circle (1pt)
		(1.5,0) circle (1pt)
		(1.75,0) circle (1pt);
		
		\filldraw (0,1.25) circle (1pt)
		(0,1.5) circle (1pt)
		(0,1.75) circle (1pt);
		
		\coordinate (A) at (3.5, 0.5);
		
		\filldraw (3,0) circle (2pt);
		\filldraw[blue] (A) circle (2pt);
		
		\node[anchor=west, blue] at (A) {\(\left( c+\frac{h_1}{\sqrt{n}}, \frac{h_2}{\sqrt{n}}\right)\)};
		\draw[->, thick, dotted, blue] (A) .. controls (3.5, 0) and (3, 0.5) .. (3, 0.1);
		
		\coordinate (B) at (0.5, 3.5);
		
		\filldraw (0,3) circle (2pt);
		\filldraw[blue] (B) circle (2pt);
		
		\node[anchor=west, blue] at (B) {\(\left( \frac{h_1}{\sqrt{n}}, c+ \frac{h_2}{\sqrt{n}}\right)\)};
		\draw[->, thick, dotted, blue] (B) .. controls (0,3.5) and (0.5,3) .. (0.1,3);
		
		\coordinate (C) at (-0.5, -0.5);
		
		\filldraw (0,0) circle (2pt);
		\filldraw[blue] (C) circle (2pt);
		
		\node[anchor=east, blue] at (C) {\(\left( \frac{h_1}{\sqrt{n}}, \frac{h_2}{\sqrt{n}}\right)\)};
		\draw[->, thick, dotted, blue] (C) --
		%.. controls (-0.5,-0.25) and (-0.25,0) ..
		(-0.1,-0.1);
	\end{tikzpicture}
	\caption[Local spaces for the mediation hypothesis]{Illustration of local spaces (in yellow) arise from the mediation hypothesis. In black: the corresponding local null spaces; in blue: examples for local alternatives}
	\label{fig:localspaces}
\end{figure}
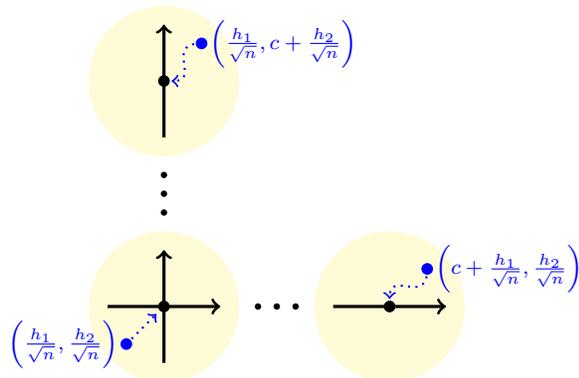

In some settings, one can explicitly show that the two-stage procedure controls the FWER.
Assume that we use an adaptive adjusting method (after filtration step, and after computing \(F\), number of unfiltered hypotheses), such that under the null \(P_{\theta}\left(\baserejeventadj \right) = \alpha/F \). For the special case that \(T\) is a p-value, it can be achieved for example by multiplying the rejection threshold by \(1/F\) (similarly to Bonferroni).
The following example, inspired by \citet{huang_genome-wide_2019}, is such an example.  
\begin{example}
    Assume that statistics follow the normal distribution with \(\estim{1} \sim \nordist[\regcoef{1}[][n]][n^{-1/2} \sigma]\) and \(\estim{2} \sim \nordist[\regcoef{2}[][n]][n^{-1/2} \sigma]\), for some constant \(\sigma\), and the means are distributed as follows:
    \begin{itemize}
        \item With probability \(\pi_0\) (Case~\nullcase{00}), \(\regcoef{1}[][n] = \regcoef{2}[][n] = 0\).
        \item With probability \(\pi_1\) (Case~\nullcase{10}), \(\regcoef{1}[][n] \sim \nordist[1 + n^{-1/2}][n^{-1/2}]\) and \(\regcoef{2}[][n] = 0\).
        \item With probability \(\pi_2\) (the alternative), \(\regcoef{1}[][n] \sim \nordist[1 + n^{-1/2}][n^{-1/2}]\) and \(\regcoef{2}[][n] \sim \nordist[n^{-1/2}][n^{-1/2}]\).
    \end{itemize}
    Here \(\pi_0 + \pi_1 + \pi_2 = 1\) are unknown. 
    Let us have \(m\) true null hypotheses, tested with the two-stage procedure, where an hypothesis is filtered (at the first stage) if \(\left|\estim{1} \estim{2}\right| < cn^{-\delta}\), for some predefined \(c > 0\), \(\frac{1}{2} < \delta < 1\).
    In this scenario, with adjusting as proposed in Theorem~\ref{thm:fwerbound},
    \begin{align*}
        P\left( V \geq 1 \right)
        &\leq \sum_{i=1}^{m} P\left( \baserejeventadj[i] \andtext \nfilevent[i]  \right) \\
        &= \sum_{i=1}^{m} \big(\frac{\pi_0}{\pi_0 + \pi_1} P\left( \cond{\baserejeventadj[i] \andtext \nfilevent[i]}{\text{Case~\nullcase{00}}} \right) \\
        &\quad+ \frac{\pi_1}{\pi_0 + \pi_1} P\left( \cond{\baserejeventadj[i] \andtext \nfilevent[i]}{\text{Case~\nullcase{10}}} \right) \big)\\
        &= \sum_{i=1}^{m} \frac{\pi_1}{\pi_0 + \pi_1} P\left( \cond{\baserejeventadj[i]}{\text{Case~\nullcase{10}}} \right) \fullstop
    \end{align*}
    Now, since \(F \to \frac{\pi_1}{\pi_0 + \pi_1}\), and \(P\left( \baserejeventadj[i]\right) = \frac{1}{F} P\left( \baserejevent[i]\right) \), we have FWER no more than the predefined level \(P\left( \baserejevent[i]\right) = \alpha\).
\end{example}

% =============================================
%   _____           _   _                 __  
%   / ____|         | | (_)               / /  
%  | (___   ___  ___| |_ _  ___  _ __    / /_  
%   \___ \ / _ \/ __| __| |/ _ \| '_ \  | '_ \ 
%   ____) |  __/ (__| |_| | (_) | | | | | (_) |
%  |_____/ \___|\___|\__|_|\___/|_| |_|  \___/ 
% =============================================

\section{Applications to Multiple Comparison} \label{appmc}

Consider the multiple tests framework discussed in Sections \ref{sec:filt-alg} and \ref{sec:la-multmedhyp}. In this section we aim to compare the proposed two-stage procedure through simulations. We present three settings. The first is similar to the setting of \citet{djordjilovic_optimal_2020}. We simulate the estimators \(\estim{1}\) and \(\estim{2}\) for the regression coefficients in Section \ref{sec:rates-framework}, from normal distribution with means \(\regcoef{1}[][n]\) and \(\regcoef{2}[][n]\) and standard deviations \(\frac{1}{\sqrt{n}} \sigma\).
The true means \(\regcoef{1}[][n]\) and \(\regcoef{2}[][n]\) are distributed as detailed in the configurations columns in Table \ref{tbl:param}.
\begin{table}[ht]
\centering
\begin{tabular}{rlllll}
  \hline
 & \(\gamma_n\) & \(\beta_n\) & Configuration 1 & Configuration 2 & Configuration 3 \\ 
  \hline
1 & \(0\) & \(0\) & \(65\%\) & \(25\%\) & \(25\%\) \\ 
  2 & \(3 n^{-3/4}\) & \(0\) & \(0\%\) & \(15\%\) & \(0\%\) \\ 
  3 & \(3 n^{-1/2}\) & \(0\) & \(30\%\) & \(25\%\) & \(35\%\) \\ 
  4 & \(3 n^{-1/3}\) & \(0\) & \(0\%\) & \(10\%\) & \(0\%\) \\ 
  5 & \(1 + 3 n^{-1/2}\) & \(0\) & \(0\%\) & \(15\%\) & \(15\%\) \\ 
  6 & \(3 n^{-1/2}\) & \(3 n^{-1/2}\) & \(5\%\) & \(4\%\) & \(0\%\) \\ 
  7 & \(3 n^{-1/3}\) & \(3 n^{-1/2}\) & \(0\%\) & \(3\%\) & \(0\%\) \\ 
  8 & \(1 + 3 n^{-1/2}\) & \(3 n^{-1/2}\) & \(0\%\) & \(3\%\) & \(10\%\) \\ 
   \hline
\end{tabular}
\caption{Parameters} 
\label{tbl:param}
\end{table}

Following Algorithm~\ref{alg:filtration}, we first use filtration tests. We compare numerous filtration methods, where the filtering event \(A\) is each of the following:%These tests are produced relative to the filtration null hypothesis, \(\theta = 0\), and the tests compared are:
\begin{itemize}
	\item \(A = \left[\min\left\lbrace \regpval{1}, \regpval{2} \right\rbrace \geq 0.0004\right] \)
	\item \(A = \left[ F_{\chi_2^2}^{-1}\left( \estim{1}^2 + \estim{2}^2\right) \geq 0.001 \right] \), where \(F_{\chi_2^2}\) is the \(\chi_2^2\)-distribution CDF.
	\item \(A = \left[ \product < cn^{-\delta}\right] \), for \(\delta = 0.8, 0.9 \andtext 1\) and \(c = 1.2, 2, 3\) respectively, similar to the filtration used in the Case Study in Section \ref{sec:la-casestudy}.
\end{itemize}
We retain each of the hypotheses that its filtration event \(A\) is true. We are then left with \(F\) hypotheses that were not filtered. In all methods, we then compute for every remaining hypothesis another \pvalue{}, \(\max\left\lbrace \regpval{1}, \regpval{2} \right\rbrace\) for the original mediation hypotheses, as in \eqref{eq:medhypmulti}. We adjust the remaining \pvalue{}s with Bonferroni method, and then reject hypotheses whose \pvalue{}s are below a certain threshold. In each test there are 200 hypotheses, and we repeated the simulation 500 times. The FWER bounds in Theorems \ref{thm:fwercontrol} and \ref{thm:fwerbound} are conservative, and following \citet{djordjilovic_optimal_2020}, we conducted the base test with Bonferroni correction, without considering the filtration probability.

In all configurations, the only change is the filtration \pvalue{}, and the base test \pvalue{}s and adjusting method are the same (beside the exact adjustment performed, which depend on the set of remaining hypotheses, after the filtration). In Figure~\ref{fig:simulatedfwerpower} are shown the FWER and the power of the different two-stage methods. An R package that implements the different testing procedures and the examples presented above can be found at \url{https://github.com/yotamleibovici/twostageshrink}.
%The simulations were performed with three types of alternative hypotheses:
%\begin{enumerate}
%	\item Parameter points which are approaching to the null point \(\left(0,0 \right) \) (in some rate), see Figures...;
%	\item Parameter points which are approaching to the null point \(\left(c,0 \right) \) or \(\left(0,c \right)\) (in some rate), see Figures...;
%	\item Alternatives which are mixed of the two above.
%\end{enumerate}

%\begin{figure}
%	\centering
%	\includegraphics[width=\textwidth]{pvals/pvals-fwer.pdf}
%%	\begin{subfigure}{0.45\textwidth}
%%		\centering
%%		\includegraphics[width=\textwidth]{pvals/pvals-fwer1.pdf}
%%		\caption{config 1}
%%	\end{subfigure}
%%	\begin{subfigure}{0.45\textwidth}
%%		\centering
%%		\includegraphics[width=\textwidth]{pvals/pvals-fwer2.pdf}
%%		\caption{config 2}
%%	\end{subfigure}
%	\caption{FWER}
%\end{figure}
%
%\begin{figure}
%	\centering
%	\includegraphics[width=\textwidth]{pvals/pvals-power.pdf}
%%	\begin{subfigure}{0.45\textwidth}
%%		\centering
%%		\includegraphics[width=\textwidth]{pvals/pvals-power1.pdf}
%%		\caption{config 1}
%%	\end{subfigure}
%%	\begin{subfigure}{0.45\textwidth}
%%		\centering
%%		\includegraphics[width=\textwidth]{pvals/pvals-power2.pdf}
%%		\caption{config 2}
%%	\end{subfigure}
%	\caption{Power}
%\end{figure}

\begin{figure}
	\centering
	\includegraphics[width=\textwidth]{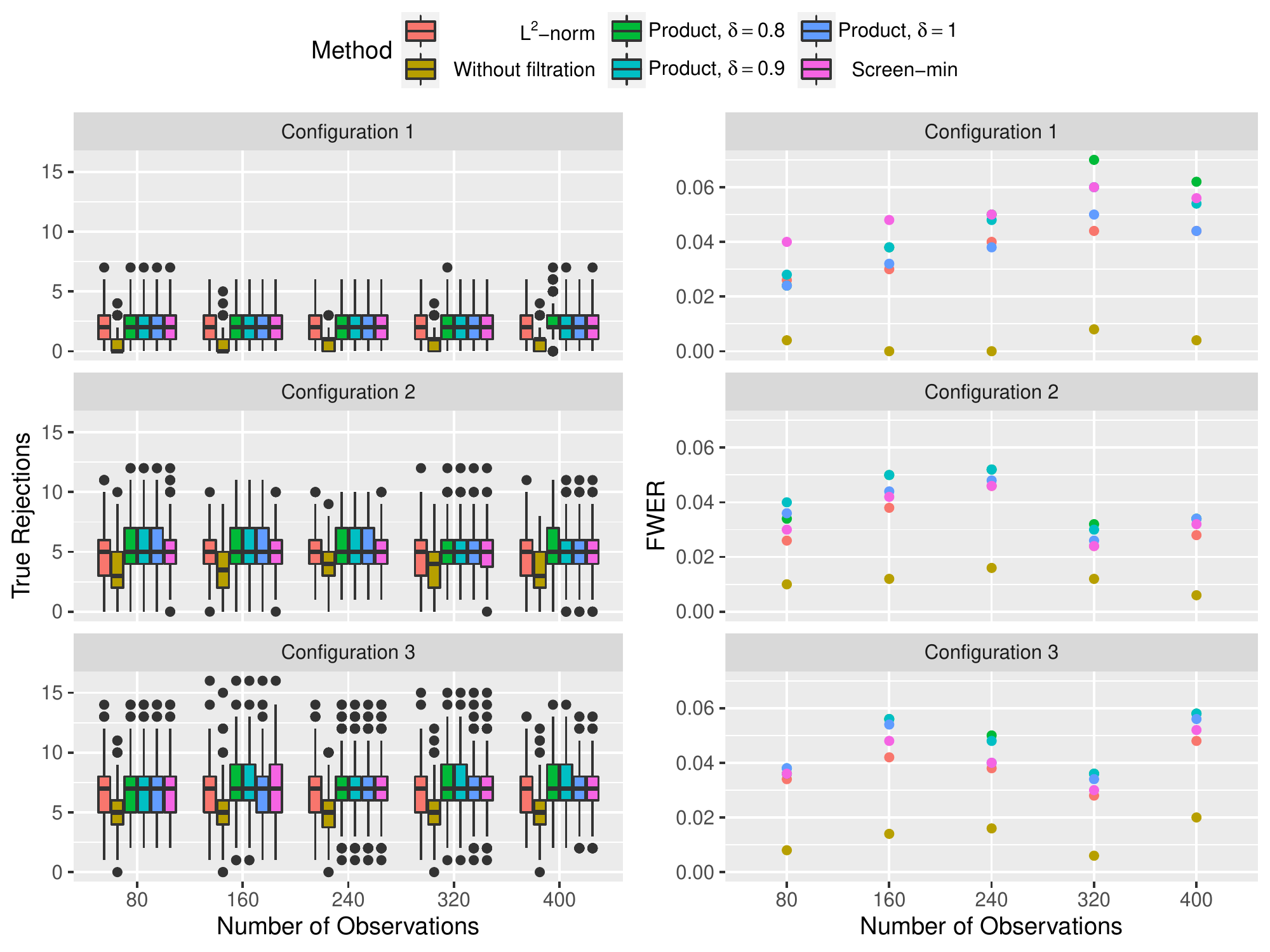}
	\caption[Empirical FWER and power]{Empirical FWER and Power of the two-stage procedure, with various filtration tests, for the multiple mediation hypotheses framework. Simulated under three different settings, as detailed in Section \ref{appmc}.}
	\label{fig:simulatedfwerpower}
\end{figure}

In all methods the simulated FWER was approximately under 0.05. In both configurations, there was a significant improvement of the filration procedures compared to the procedure without filtering at all. In Configuration 2, there is an improvement of the product methods with limit probability 1 (i.e., \(\delta = 0.8 \andtext 0.9\)) above the rest of the filtration procedures, that are all have similar results in terms of power.

\section{Concluding Remarks}
In this work, we considered the mediation problem arises often and in large numbers in genome-wide research, and the power-loss challenge involves it. We examined the two-stage approach for testing the mediation problem, which includes a filtration test prior to some base test. We linked the two-stage procedure to shrinkage estimators, and studied their properties in the local asymptotics framework, which is essential in this context of composite hypotheses with filtration. We also stated result concerning FWER control for the multiple hypotheses framework, using similar terminology from the point-estimation framework. We then demonstrated the theoretical results about estimators through a case study, and performed simulation study for the multiple hypotheses framework, which shows superiority of the filtration methods, and for some scenarios, of some of them on the others.

%%%%%%%%%%%%%%%%%%%%%%%%%%%%%%%%%%%%%%%%%%%%%%
%% Supplementary Material, if any, should   %%
%% be provided in {supplement} environment  %%
%% with title and short description.        %%
%%%%%%%%%%%%%%%%%%%%%%%%%%%%%%%%%%%%%%%%%%%%%%
%\begin{supplement}
%\stitle{???}
%\sdescription{???.}
%\end{supplement}

%% if your bibliography is in bibtex format, uncomment commands:
%\bibliographystyle{imsart-number} % Style BST file (imsart-number.bst or imsart-nameyear.bst)
%\bibliography{msc-thesis-paper--ref}       % Bibliography file (usually '*.bib')

%% or include bibliography directly:
% \begin{thebibliography}{}
% \bibitem{b1}
% \end{thebibliography}

% AOS {
	%%%%%%%%%%%%%%%%%%%%%%%%%%%%%%%%%%%%%%%%%%%%%%
	%% Please use \tableofcontents for articles %%
	%% with 50 pages and more                   %%
	%%%%%%%%%%%%%%%%%%%%%%%%%%%%%%%%%%%%%%%%%%%%%%
	%\tableofcontents
	
	%%%%%%%%%%%%%%%%%%%%%%%%%%%%%%%%%%%%%%%%%%%%%%
	%%%% Main text entry area:

	%%%%%%%%%%%%%%%%%%%%%%%%%%%%%%%%%%%%%%%%%%%%%%
	%% Single Appendix:                         %%
	%%%%%%%%%%%%%%%%%%%%%%%%%%%%%%%%%%%%%%%%%%%%%%
\begin{appendix}
	%\section*{???}%% if no title is needed, leave empty \section*{}.
\section*{Proofs} \label{sec:proofs}

\begin{proof}[Proof of Lemma~\ref{lem:irregularity}]
	For Sobel's test statistic \(\frac{\product}{\sqrt{\sigma_{\regcoef{2}}^2 \estim{1}^2 + \sigma_{\regcoef{1}}^2 \estim{2}^2}}\), let \(\regcoef{1} = \frac{h_{\regcoef{1}}}{\sqrt{n}}\) and \(\regcoef{2} = \frac{h_{\regcoef{2}}}{\sqrt{n}}\), and consider it as an estimator for \(\frac{\regcoef{1}\regcoef{2}}{\sqrt{\sigma_{\regcoef{2}}^2 \regcoef{1}^2 + \sigma_{\regcoef{1}}^2 \regcoef{2}^2}}\). Now,
	\begin{align*}
	\sqrt{n} \left( \frac{\product}{\sqrt{\sigma_{\regcoef{2}}^2 \estim{1}^2 + \sigma_{\regcoef{1}}^2 \estim{2}^2}} - \frac{\regcoef{1}\regcoef{2}}{\sqrt{\sigma_{\regcoef{2}}^2 \regcoef{1}^2 + \sigma_{\regcoef{1}}^2 \regcoef{2}^2}} \right)
	=  \frac{\sqrt{n}\estim{1}\sqrt{n}\estim{2}}{\sqrt{\sigma_{\regcoef{2}}^2 n \estim{1}^2 + \sigma_{\regcoef{1}}^2 n \estim{2}^2}} - \frac{h_{\regcoef{1}}h_{\regcoef{2}}}{\sqrt{\sigma_{\regcoef{2}}^2 h_{\regcoef{1}}^2 + \sigma_{\regcoef{1}}^2 h_{\regcoef{2}}^2}} \\
	\to
	\frac{%
	\left( h_{\regcoef{1}} + Z_{\regcoef{1}} \right)
	\left( h_{\regcoef{2}} + Z_{\regcoef{2}} \right)
	}{%
	\sqrt{\sigma_{\regcoef{2}}^2 \left( h_{\regcoef{1}} + Z_{\regcoef{1}} \right)^2 +
	\sigma_{\regcoef{1}}^2 \left( h_{\regcoef{1}} + Z_{\regcoef{1}} \right)^2}
	} - \frac{h_{\regcoef{1}}h_{\regcoef{2}}}{\sqrt{\sigma_{\regcoef{2}}^2 h_{\regcoef{1}}^2 + \sigma_{\regcoef{1}}^2 h_{\regcoef{2}}^2}} \comma
	\end{align*}
	For \(Z_{\regcoef{1}}\) and \(Z_{\regcoef{2}}\) independent standard normal random variables. Clearly, for different values of \(\left(h_{\regcoef{1}}, h_{\regcoef{1}} \right) \), for instance for \(\left(0,0 \right) \) (which represent a constant series at the origin) and \(\left(1,0 \right) \) (which represent converging from the positive side of the \(x\)-axis to the origin), the distributions obtained are different. Similar arguments can be applied also for \(\sqrt{\estim{1}^2 + \estim{2}^2}\).
\end{proof}

	\begin{lemma}
	Let \(X_n \geq 0\) be a sequence of non-negative uniformly integrable random variables, and \(A_n\) sets such that \(P\left(A_n \right) \to L \), \(0<L<1\). Then \[0 < \liminf \mean{X_n \indc*{A_n}} <
	\limsup \mean{X_n \indc*{A_n}} < 1\]
\end{lemma}

\begin{proof}
	The sets \(A_n\) are eventually (for \(n>N\) and some \(N\)) of probability \(P\left(A_n \right) > L - \varepsilon \) for some (any) \(0 < \varepsilon < L\). Let \(B_n \subset A_n\) such that \(P\left(B_n \right) = L - \varepsilon \), for \(n>N\). It holds that for \(n > N\)
	\begin{align}
		\mean{X_n \indc*{A_n}} > \mean{X_n \indc*{B_n}} %\geq 
		%\mean{X_n \indc{X_n \leq F_{X_n}^{-1}\left( L - \varepsilon\right) }} 
		\fullstop
	\end{align}
	Assume by contradiction that the right hand side had partial limit \(0\). Then there was an indices series \(m_n\) such that for any \(\delta > 0\),
	\(P\left(\cond{
		X_{m_n} < \varepsilon}{B_n} \right) \to 1 \)
	(otherwise, )
	Thus, for every \(\delta > 0\), \(\lim P\left(X_{m_n} < \varepsilon \right) > L -\varepsilon  \), in contradiction to the convergence to normal distribution.
\end{proof}

	\begin{proof}[Proof of Theorem~\ref{thm:effclass}]
	Consider the ratio
	\begin{align*}
		\frac{\MSE\left(\estimfil; \theta\right) }{\MSE\left(\estimreg; \theta\right)} &=
		\frac{\var{\estimfil} + \mean{\estimfil - \theta_n }^2 }{\var{\estimreg}} \cdot \frac{\var{\estimreg}}{\var{\estimreg} + \mean{\estimreg - \theta_n }^2} \\&= \frac{\var{\estimfil} + \mean{\estimfil - \theta_n }^2 }{\var{\estimreg}}(1-o(1)) \comma
	\end{align*}
	where the last equality holds since, by Assumption \ref{assum:asymunbiased}, \[\frac{\var{\estimreg}}{\var{\estimreg} + \mean{\estimreg - \theta_n }^2} = \frac{\var{r_n\estimreg}}{\var{r_n\estimreg} + \mean{r_n\left( \estimreg - \theta_n \right) }^2} = 1 - o(1) \fullstop\]
	
	When \(L = 1\),
	%		the variance ratio is
	%		\begin{align}\begin{split}\label{eq:varratio}
	%		\frac{\var{\estimfil}}{\var{\estimreg}}
	%		&= \frac{\var{\estimfil - \psi\left(\theta_n\right) }}{\var{\estimreg}} 
	%		\leq  \mean{\frac{1}{\sigma_{\estimreg}^2}\left(\estimreg - \psi\left(\theta_n\right) \right)^2 \indc{\nfilevent} }
	%		\to 0 \comma
	%		\end{split}\end{align}
	
	%		For the other summand,
	\begin{equation} \label{eq:decompratio}
		\frac{\mean{\left(  \estimfil - \theta_n \right)^2 } }{\var{\estimreg}} =
		\mean{\frac{1}{\sigma_{\estimreg}^2} \left( \estimreg - \theta_n \right)^2 \indc{\nfilevent[][n]} }+\mean{\frac{1}{\sigma_{\estimreg}^2} \left( \theta_0 - \theta_n \right)^2 \indc{\filevent[][n]} }
	\end{equation}
	The first summand converges to \(0\) since the sequence \(\frac{1}{\sigma_{\estimreg}^2}\left(\estimreg - \psi\left(\theta_n\right) \right)^2\) is uniformly integrable by Assumption~\ref{assum:asymunbiased}. The second summand converges to \(K\) by assumption. Thus,
	\begin{align}
		\frac{\MSE\left(\estimfil; \theta\right) }{\MSE\left(\estimreg; \theta\right)} \to \left( 0 + K\right) ^2 \cdot \left(1-o\left(1 \right)  \right) = K^2
	\end{align}
	Consequently, every one of the relative efficiency possibilities may occur, depending directly on the value of \(K\).

	When \(0<L<1\), the MSE-ratio can be decomposed as in \eqref{eq:decompratio}. This time, the second summand is
	\[\mean{\frac{1}{\sigma_{\estimreg}^2} \left( \theta_0 - \theta_n \right)^2 \indc{\nfilevent[][n]} } \to K^2 L\comma\]
	For the first summand,
	%		the variance ratio
	%		\begin{align}
	%		\frac{\var{\estimfil}}{\var{\estimreg}}
	%		= \frac{\var{\estimfil - \psi\left(\theta_n\right) }}{\var{\estimreg}} 
	%		\leq \mean{\frac{1}{\sigma_{\estimreg}^2}\left(\estimreg - \psi\left(\theta_n\right) \right)^2 \indc{\nfilevent} }
	%		\fullstop %\mean{\frac{1}{\sigma_{\estimreg}^2}\left(\estimreg - \psi\left(\theta_n\right) \right)^2}\to 1 \comma
	%		\end{align}
	with no further assumption, it can be bounded in the limit by \(0\) and \(1\), but not strictly. Sharp inequalities can be achieved with Lemma A.1, assuming all the following limits exist:
	\begin{align*}
	    \lim \mean{\frac{1}{\sigma_{\estimreg}^2}\left(\estimreg - \psi\left(\theta_n\right) \right)^2 \indc{\nfilevent[][n]} } &=\\ \lim \mean{\frac{1}{\sigma_{\estimreg}^2}\left(\estimreg - \psi\left(\theta_n\right) \right)^2 } &- \mean{\frac{1}{\sigma_{\estimreg}^2}\left(\estimreg - \psi\left(\theta_n\right) \right)^2 \indc{\filevent[][n]} } < 1 \comma
	\end{align*}
	and symmetrically have sharp lower bound of \(0\).
	%The analysis of the second expression
	%\(\frac{1}{\sigma_{\estimreg}}\left( \mean{\estimfil - \theta_n }\right) ^2 \) is the same as in \eqref{eq:varratio}, and it converges to \(K^2\).
	When the variance ratio is indeed approaches to a constant \(0<c<1\), then the whole expression goes to \(K^2 L+c\), which means that the relative efficiency can be any of the options, besides much more efficient. Otherwise, there are case of which it is possible to gain much more efficiency, or alternatively, that even (just) more efficiency is not possible.
	
	When \(L = 0\), the MSE ratio is \(1\) since \(P\left(\estimreg = \estimfil \right) \to 1 \). 
\end{proof}

\begin{proof}[Proof of Theorem~\ref{thm:fwercontrol}]
		Formally, the assumption for the adjusting method is that \[P_{\theta}\left(\cond{\baserejeventadj}{\sum_{j\neq i}f_i} \right) = \frac{P_{\theta_0}\left( \nfilevent \right)}{\sum_{j\neq i}f_i+1} \alpha \], i.e., the adjusting of the reject region is independent of the \(i\)-th filtration event.
	
	Let \(f\) be the vector with value 1 or 0 in his \(i\)-th element with correspondence to whether \(S^i \in Q^i\) or not. We have
	\begin{align*}
		P\left(\cond{V \geq 1 }{f}\right)
		&= P\left(\cond{\bigcup_{f_i=1} \left\lbrace T^i \in R^{i,\text{adj}} \right\rbrace }{ f}  \right) \\
		&\leq \sum_{f_i=1} P\left( \cond{ T^i \in R^{i,\text{adj}}  }{ {f_i = 1} }, \sum_{j\neq i}f_i \right) \\
		&\leq \sum_{f_i=1} \frac{P\left(\cond{ T^i \in R^{i,\text{adj}} \andtext S^i \in Q^i}{\sum_{j\neq i}f_i} \right)}{P_\theta \left(\cond{S^i \in Q^i}{\sum_{j\neq i}f_i} \right)} \\
		&\leq \sum_{f_i=1} \frac{P\left( \cond{ T^i \in R^{i,\text{adj}}}{\sum_{j\neq i}f_i} \right)}{P_\theta \left(S^i \in Q^i\right)}  \stepcounter{equation}\tag{\theequation}\label{eq:indep}\\
		&\leq F \frac{P_{\theta_0}\left(\nfilevent[i] \right)}{F} \alpha \frac{1}{P_\theta \left(S^i \in Q^i\right)} \\
		&\leq \alpha \comma
	\end{align*}
	where \eqref{eq:indep} follows from the independence of the hypotheses.
	The un-conditional bound can be achieved by taking an expectation over \(f\) values. (It will be still bounded by \(\alpha\)).
\end{proof}

\begin{proof}[Proof of Theorem~\ref{thm:fwerbound}]

	Let \(f\) be the vector with value 1 or 0 in its \(i\)-th element with correspondence to whether \(S^i \in Q^i\) or not.
	
	\begin{align*}
		P\left(\cond{
			V \geq 1
		}{
			f
		}\right)
		&= 1 - P\left(\cond{V = 0}{f}\right) \\
		&= 1 - P\left( \cond{
			\bigcap_{f_i=1} \left\lbrace  T^i \notin R^{i,\text{adj}} \right\rbrace 
		}{f} \right) \\
		&= 1 - \prod_{f_i=1} P\left( \cond{
			T^{i} \notin R^{i,\text{adj}}
		}{f_i = 1, F} \right) \\
		&= 1 - \prod_{f_i=1} \left( 1- P_{\theta}\left( \cond{T^{i} \notin R^{i,\text{adj}}}{S^i \in Q^i, F} \right)\right)  \\
		&\leq 1 - \left( 1- \max_{\theta\in\Theta} P_{\theta}\left( \cond{T^{i} \in R^{i,\text{adj}}}{S^i \in Q^i , F} \right)\right) ^F
	\end{align*}
	When the final bound is achieved by taking expectation over \(f\) values.
\end{proof}

\begin{proof}[MSE-ratio for \(0<L<1\) in the Case Study \ref{sec:la-casestudy}]
	\begin{align*}
		\operatorname{MSE}\!\left( \widehat{\theta}_n \right)
		&= \mean{\!\left( \widehat{\theta}_n - \theta_n \right)^{\!2}} \\
		&= \mean{\left( \overline{X}_n \overline{Y}_n - \mu_X^n \mu_Y^n \right) ^2} \\
		&= \mean{\left( \avg{X} \!\left( \overline{Y}_n - \mu_Y^n \right) +
			\left( \avg{X} - \mu_X^n \right)\! \mu_Y^n \right) ^2}\\
		&= \mean{\left( \avg{X} \!\left( \overline{Y}_n - \mu_Y^n \right) \right)^2} +
		\mean{\left( \left( \avg{X} - \mu_X^n \right)\! \mu_Y^n \right)^2}
		\\
		&\quad +
		2\,\mean{\avg{X} \!\left( \overline{Y}_n - \mu_Y^n \right)
			\left( \avg{X} - \mu_X^n \right)\! \mu_Y^n }\\
		&=2\,\mean{\avg{X}\left( \avg{X} - \mu_X^n \right) }\underbrace{\mean{\avg{Y}-\mu_Y^n}}_0\mu_Y^n + \mean{\avg{X}^2}\mean{\left( \avg{Y}-\mu_Y^n\right) ^2} + 
		\mean{\left( \avg{X}-\mu_X^n \right) ^2 } {\mu_Y^n}^2
		\\
		&= \left( n^{-1} + {\mu_X^n}^2 \right) n^{-1} + n^{-1} {\mu_Y^n}^2 \\
		&= n^{-1} \left( n^{-1} + {\mu_X^n}^2 + {\mu_Y^n}^2 \right)
		.
	\end{align*}
	Hence,
	\begin{align*}
		\operatorname{MSE}\!{\left( \widetilde{\theta}_n\right) } -\operatorname{MSE}\!{\left( \widehat{\theta}_n\right) }
		&= \mean{\left( \widetilde{\theta}_n - \theta_n \right)^2 } - \mean{\left( \widehat{\theta}_n - \theta_n \right)^2 } \\
		&= \mean{\widetilde{\theta}_n^{\,2}-\widehat{\theta}_n^{\,2}} -
		2\theta_n\,\mean{\widetilde{\theta}_n-\widehat{\theta}_n} +
		\theta_n^2-\theta_n^2 \\
		&= \mean{-\avg{X}^2\,\avg{Y}^2 \indc{\left| \avg{X}\avg{Y} \right| \le cn^{-\gamma} }} +
		2\theta_n \mean{\avg{X}\avg{Y}\indc{\left| \avg{X}\avg{Y} \right| \le cn^{-\gamma}}} \\
		&= \mean{\left( 2\theta_n\avg{X}\avg{Y} - \avg{X}^2\,\avg{Y}^2 \right)
			\indc{\left| \avg{X}\avg{Y} \right| \le cn^{-\gamma}}} \\
		&\le \mean{\left( 2\theta_n\avg{X}\avg{Y} \right)
			\indc{\left| \avg{X}\avg{Y} \right| \le cn^{-\gamma}}} \\
		&\le 2\theta_n \cdot cn^{-\gamma} P\left( \left| \avg{X}\avg{Y} \right| \le cn^{-\gamma} \right) \\
		&= 2c\,\mu_X^n \mu_Y^n \, n^{-\gamma} P\left( \left| \avg{X}\avg{Y} \right| \le cn^{-\gamma} \right)
		.
	\end{align*}
	
	Therefore,
	\begin{align*}
		\frac{\operatorname{MSE}\!{\left( \widetilde{\theta}_n\right) } -\operatorname{MSE}\!{\left(
				\widehat{\theta}_n\right) }}
		{\operatorname{MSE}\!\left( \widehat{\theta}_n \right)}
		&\le \frac
		{2c\,\mu_X^n \mu_Y^n \, n^{-\gamma} P\left( \left| \avg{X}\avg{Y} \right| \le cn^{-\gamma} \right)}
		{n^{-1} \left( n^{-1} + {\mu_X^n}^2 + {\mu_Y^n}^2 \right)} \\
		&\le \frac{2c\,\mu_X^n\mu_Y^n} {n^{-1}+{\mu_X^n}^2+{\mu_Y^n}^2} n^{1-\gamma} \\
		&\le \frac{ c\left( {\mu_X^n}^2 + {\mu_Y^n}^2\right) } {n^{-1}+{\mu_X^n}^2+{\mu_Y^n}^2} n^{1-\gamma}
		.
	\end{align*}
	%When $n\left( {\mu_X^n}^2 + {\mu_Y^n}^2 \right) \to 0$, the expression above tends to $0$. When $n\left( {\mu_X^n}^2 + {\mu_Y^n}^2 \right) \to \infty$, for instance, if the two expectations are constants, this expression tends to $c$. Finally,
	When \(\gamma = 1\) and ${\mu_X^n}^2 + {\mu_Y^n}^2 \sim n^{-1}$, for instance when \(n^{0.5}\left(\regcoef{1}[n], \regcoef{2}[n] \right) \to h \), \(0<h<\infty\), the expression tends to \(1+c\).

\end{proof}
\end{appendix}

	\bibliographystyle{plainnat}
    \bibliography{msc-thesis-paper--ref}

\begin{thebibliography}{25}
\providecommand{\natexlab}[1]{#1}
\providecommand{\url}[1]{\texttt{#1}}
\expandafter\ifx\csname urlstyle\endcsname\relax
  \providecommand{\doi}[1]{doi: #1}\else
  \providecommand{\doi}{doi: \begingroup \urlstyle{rm}\Url}\fi

\bibitem[Agha et~al.(2015)Agha, Houseman, Kelsey, Eaton, Buka, and
  Loucks]{agha_adiposity_2015}
Golareh Agha, E~Andres Houseman, Karl~T Kelsey, Charles~B Eaton, Stephen~L
  Buka, and Eric~B Loucks.
\newblock Adiposity is associated with {DNA} methylation profile in adipose
  tissue.
\newblock \emph{International Journal of Epidemiology}, 44\penalty0
  (4):\penalty0 1277--1287, 2015.

\bibitem[Barfield et~al.(2017)Barfield, Shen, Just, Vokonas, Schwartz,
  Baccarelli, VanderWeele, and Lin]{barfield_testing_2017}
Richard Barfield, Jincheng Shen, Allan~C. Just, Pantel~S. Vokonas, Joel
  Schwartz, Andrea~A. Baccarelli, Tyler~J. VanderWeele, and Xihong Lin.
\newblock Testing for the indirect effect under the null for genome-wide
  mediation analyses.
\newblock \emph{Genetic Epidemiology}, 41\penalty0 (8):\penalty0 824--833,
  December 2017.

\bibitem[Berger(1982)]{berger_multiparameter_1982}
Roger~L. Berger.
\newblock Multiparameter {Hypothesis} {Testing} and {Acceptance} {Sampling}.
\newblock \emph{Technometrics}, 24\penalty0 (4):\penalty0 295, November 1982.

\bibitem[Berger and Hsu(1996)]{berger_bioequivalence_1996}
Roger~L. Berger and Jason~C. Hsu.
\newblock Bioequivalence trials, intersection-union tests and equivalence
  confidence sets.
\newblock \emph{Statistical Science}, 11\penalty0 (4):\penalty0 283--319, 1996.

\bibitem[Borghol et~al.(2012)Borghol, Suderman, McArdle, Racine, Hallett,
  Pembrey, Hertzman, Power, and Szyf]{borghol_associations_2012}
Nada Borghol, Matthew Suderman, Wendy McArdle, Ariane Racine, Michael Hallett,
  Marcus Pembrey, Clyde Hertzman, Chris Power, and Moshe Szyf.
\newblock Associations with early-life socio-economic position in adult {DNA}
  methylation.
\newblock \emph{International Journal of Epidemiology}, 41\penalty0
  (1):\penalty0 62--74, 2012.

\bibitem[Bourgon et~al.(2010)Bourgon, Gentleman, and
  Huber]{bourgon_independent_2010}
R.~Bourgon, R.~Gentleman, and W.~Huber.
\newblock Independent filtering increases detection power for high-throughput
  experiments.
\newblock \emph{Proceedings of the National Academy of Sciences}, 107\penalty0
  (21):\penalty0 9546--9551, May 2010.

\bibitem[Djordjilović et~al.(2020)Djordjilović, Hemerik, and
  Thoresen]{djordjilovic_optimal_2020}
Vera Djordjilović, Jesse Hemerik, and Magne Thoresen.
\newblock On optimal two-stage testing of multiple mediators.
\newblock \emph{arXiv:2007.02844 [stat]}, 2020.

\bibitem[Hackstadt and Hess(2009)]{hackstadt_filtering_2009}
Amber~J Hackstadt and Ann~M Hess.
\newblock Filtering for increased power for microarray data analysis.
\newblock \emph{BMC Bioinformatics}, 10\penalty0 (1):\penalty0 11, 2009.

\bibitem[Huang(2019)]{huang_genome-wide_2019}
Yen-Tsung Huang.
\newblock Genome-wide analyses of sparse mediation effects under composite null
  hypotheses.
\newblock \emph{The Annals of Applied Statistics}, 13\penalty0 (1):\penalty0
  60--84, 2019.

\bibitem[Ignatiadis et~al.(2016)Ignatiadis, Klaus, Zaugg, and
  Huber]{ignatiadis_data-driven_2016}
Nikolaos Ignatiadis, Bernd Klaus, Judith~B Zaugg, and Wolfgang Huber.
\newblock Data-driven hypothesis weighting increases detection power in
  genome-scale multiple testing.
\newblock \emph{Nature Methods}, 13\penalty0 (7):\penalty0 577--580, July 2016.

\bibitem[James and Stein(1992)]{kotz_estimation_1992}
W.~James and Charles Stein.
\newblock Estimation with {Quadratic} {Loss}.
\newblock In Samuel Kotz and Norman~L. Johnson, editors, \emph{Breakthroughs in
  {Statistics}}, pages 443--460. Springer New York, 1992.

\bibitem[Le~Cam(1953)]{le_cam_asymptotic_1953}
Lucien~M Le~Cam.
\newblock \emph{On some asymptotic properties of maximum likelihood estimates
  and related {Bayes}' estimates.}
\newblock University of California Press, Berkeley, 1953.

\bibitem[Liu and Berger(1995)]{liu_uniformly_1995}
Huimei Liu and Roger~L. Berger.
\newblock Uniformly {More} {Powerful}, {One}-{Sided} {Tests} for {Hypotheses}
  {About} {Linear} {Inequalities}.
\newblock \emph{Annals of Statistics}, 23\penalty0 (1):\penalty0 55--72, 1995.

\bibitem[MacKinnon(2012)]{mackinnon_introduction_2012}
David~P. MacKinnon.
\newblock \emph{Introduction to {Statistical} {Mediation} {Analysis}}.
\newblock Routledge, 2012.

\bibitem[MacKinnon et~al.(2002)MacKinnon, Lockwood, Hoffman, West, and
  Sheets]{mackinnon_comparison_2002}
David~P. MacKinnon, Chondra~M. Lockwood, Jeanne~M. Hoffman, Stephen~G. West,
  and Virgil Sheets.
\newblock A comparison of methods to test mediation and other intervening
  variable effects.
\newblock \emph{Psychological Methods}, 7\penalty0 (1):\penalty0 83--104, 2002.

\bibitem[McClintick and Edenberg(2006)]{mcclintick_effects_2006}
Jeanette~N. McClintick and Howard~J. Edenberg.
\newblock Effects of filtering by {Present} call on analysis of microarray
  experiments.
\newblock \emph{BMC bioinformatics}, 7:\penalty0 49, 2006.

\bibitem[Pearl(2001)]{pearl_direct_2001}
Judea Pearl.
\newblock Direct and indirect effects.
\newblock In \emph{Proceedings of the {Seventeenth} conference on {Uncertainty}
  in artificial intelligence}, pages 411--420. Morgan Kaufmann Publishers Inc.,
  2001.

\bibitem[Pearl(2009)]{pearl_causal_2009}
Judea Pearl.
\newblock Causal inference in statistics: {An} overview.
\newblock \emph{Statistics Surveys}, 3\penalty0 (0):\penalty0 96--146, 2009.

\bibitem[Robins and Greenland(1992)]{robins_identifiability_1992}
J.~M. Robins and S.~Greenland.
\newblock Identifiability and exchangeability for direct and indirect effects.
\newblock \emph{Epidemiology (Cambridge, Mass.)}, 3\penalty0 (2):\penalty0
  143--155, 1992.

\bibitem[Senese et~al.(2009)Senese, Almeida, Fath, Smith, and
  Loucks]{senese_associations_2009}
L.~C. Senese, N.~D. Almeida, A.~K. Fath, B.~T. Smith, and E.~B. Loucks.
\newblock Associations {Between} {Childhood} {Socioeconomic} {Position} and
  {Adulthood} {Obesity}.
\newblock \emph{Epidemiologic Reviews}, 31\penalty0 (1):\penalty0 21--51, 2009.

\bibitem[Sobel(1982)]{sobel_asymptotic_1982}
Michael~E. Sobel.
\newblock Asymptotic {Confidence} {Intervals} for {Indirect} {Effects} in
  {Structural} {Equation} {Models}.
\newblock \emph{Sociological Methodology}, 13:\penalty0 290, 1982.

\bibitem[Talloen et~al.(2007)Talloen, Clevert, Hochreiter, Amaratunga, Bijnens,
  Kass, and Göhlmann]{talloen_ini-calls_2007}
Willem Talloen, Djork-Arné Clevert, Sepp Hochreiter, Dhammika Amaratunga, Luc
  Bijnens, Stefan Kass, and Hinrich~W.H. Göhlmann.
\newblock I/{NI}-calls for the exclusion of non-informative genes: a highly
  effective filtering tool for microarray data.
\newblock \emph{Bioinformatics}, 23\penalty0 (21):\penalty0 2897--2902,
  November 2007.

\bibitem[van~der Vaart(1998)]{van_der_vaart_asymptotic_1998}
A.~W. van~der Vaart.
\newblock \emph{Asymptotic {Statistics}}.
\newblock Cambridge University Press, 1 edition, 1998.

\bibitem[VanderWeele(2013)]{vanderweele_three-way_2013}
Tyler~J. VanderWeele.
\newblock A {Three}-way {Decomposition} of a {Total} {Effect} into {Direct},
  {Indirect}, and {Interactive} {Effects}:.
\newblock \emph{Epidemiology}, 24\penalty0 (2):\penalty0 224--232, 2013.

\bibitem[Vanderweele and Vansteelandt(2009)]{vanderweele_conceptual_2009}
Tyler~J. Vanderweele and Stijn Vansteelandt.
\newblock Conceptual issues concerning mediation, interventions and
  composition.
\newblock \emph{Statistics and Its Interface}, 2\penalty0 (4):\penalty0
  457--468, 2009.

\end{thebibliography}

\end{document}